\documentclass[12pt]{article}

\usepackage{array}
\usepackage{amssymb}
\usepackage{amsmath,amsfonts}
\usepackage{amsthm,bbm,eufrak,upgreek}
\usepackage{mathrsfs,booktabs}
\usepackage{amsxtra,mathtools}
\usepackage[colorlinks=True,linkcolor=blue,anchorcolor=blue,citecolor=blue,filecolor=blue,CJKbookmarks=True]{hyperref}
\usepackage[a4paper,left=2.6cm,right=2.6cm,top=3.2cm,bottom=3.2cm]{geometry}

\usepackage[hang,flushmargin]{footmisc}

\numberwithin{equation}{section}

\newcommand{\1}{\mathbbm{1}}
\newcommand{\Z}{{\mathbb Z}}
\newcommand{\C}{{\mathbb C}}

\newcommand{\si}{\sigma}

\newcommand{\be}{\beta}

\newcommand{\la}{\langle}
\newcommand{\ra}{\rangle}

\DeclareMathOperator{\Aut}{Aut}

\DeclareMathOperator{\End}{End}

\DeclareMathOperator{\wt}{wt}
\DeclareMathOperator{\Hom}{Hom}
\DeclareMathOperator{\Irr}{Irr}

\newtheorem{thm}{Theorem}[section]
\newtheorem{prop}[thm]{Proposition}
\newtheorem{lem}[thm]{Lemma}
\newtheorem{cor}[thm]{Corollary}
\newtheorem{rmk}[thm]{Remark}
\newtheorem{defn}[thm]{Definition}

\begin{document}

\begin{center}
{\Large \bf  Representations and fusion rules for the orbifold vertex operator algebras $L_{\widehat{\frak{sl}_2}}(k,0)^{\mathbb{Z}_3}$}
\end{center}

\begin{center}
{
% Cuipo Jiang
% \footnote{Supported by China NSF grants No.11771281 and No.11531004. email: cpjiang@sjtu.edu.cn.}
% and 
Bing Wang
\let\thefootnote\relax\footnotetext{Supported by China NSF grant No.11771281 and SNSFC grant No.16ZR1417800. \\     Email: ering123@sjtu.edu.cn.}\\
School of Mathematical  Sciences, Shanghai Jiao Tong University\\
Shanghai 200240, China}
\end{center}

\begin{abstract}
    For the cyclic group $\mathbb{Z}_3$ and positive integer $k$, we study the representations of the orbifold vertex operator algebra $L_{\widehat{\mathfrak{sl}_2}}(k,0)^{\mathbb{Z}_3}$. All the irreducible modules for $L_{\widehat{\mathfrak{sl}_2}}(k,0)^{\mathbb{Z}_3}$ are classified and constructed explicitly. Quantum dimensions and fusion rules for the orbifold vertex operator algebra $L_{\widehat{\mathfrak{sl}_2}}(k,0)^{\mathbb{Z}_3}$ are completely determined.
\end{abstract}

% \textbf{Keywords}: commutant vertex operator algebra, orbifold vertex operator algebra, irreducible modules

\section{Introduction}

    The orbifold construction is a powerful tool for constructing new vertex algebras from given ones. Let $V$ be a vertex operator algebra and $G$ a finite group consisting of certain automorphisms of $V$, the fixed point subalgebra $V^G=\{v \in V \mid gv=v, g \in G\}$ is called an orbifold vertex operator subalgebra of $V$. Many interesting examples, especially orbifold vertex operator algebras related to affine vertex operator algebras and lattice vertex operator algebras, have been extensively studied both in the physics and mathematics literature (\cite{CM}, \cite{DJ13}, \cite{DM97}, \cite{DN99}, \cite{DRX17},\cite{DY02},\cite{JWQ19},\cite{JW2},\cite{M15},\cite{MT04}, etc.). 

    The orbifold theory is concerned with the properties and representation theory of the fixed point vertex operator subalgebra $V^G$. It is natural to ask whether $V^G$ inherits some properties from $V$, such as simplicity, rationality, $C_2$-cofiniteness and regularity. It has been established that if $V$ is a regular and selfdual vertex operator algebra of CFT type and $G$ is a finite solvable group, then $V^G$ is again a regular and selfdual vertex operator algebra of CFT type \cite{CM}, \cite{M15}. 
    The decomposition of $V$ into a direct sum of irreducible $V^G$-modules was initiated in \cite{DLM96-2} and \cite{DM97}. The decomposition of an arbitrary irreducible g-twisted $V$-module into a direct sum of $V^G$-modules was achieved in \cite{DY02} and \cite{MT04}. It follows from \cite{DRX17} that if $V^G$ is a regular and selfdual vertex operator algebra of CFT type, then any irreducible $V^G$-module occurs in an irreducible $g$-twisted $V$-module for some $g \in G$. In other words, the irreducible $V^G$-modules were completely classified if $V^G$ is a regular and selfdual vertex operator algebra of CFT type.

    This paper is prompted by the results of \cite{DJ13}. The orbifold vertex operator algebra $V_{L_2}^{A_4}$ was investigated in \cite{DJ13}, where $L_2$ is the root lattice of the simple Lie algebra $\mathfrak{sl}_2$ and $A_4$ is the alternating group which is a subgroup of the automorphism group of lattice vertex operator algebra $V_{L_2}$. The main idea is to realize $V_{L_2}^{A_4}$ as $(V_{L_8}^{+})^{\la \sigma \ra}$ where $L_8$ is a rank one lattice defined in \cite{DG98} and $\sigma$ is an automorphism of $\mathfrak{sl}_2$ of order $3$. 
    Note that $V_{L_2}$ is isomorphic to $L_{\widehat{\mathfrak{sl}_2}}(1,0)$ as vertex operator algebras. 
    It is well known that $L_{\widehat{\mathfrak{sl}_2}}(k,0)$ is a regular and selfdual vertex operator algebra of CFT type for $k\in\Z_{\geqslant 1}$ \cite{FZ92}, \cite{LeL04}. 
    It is natural to consider the orbifold vertex operator algebra $L_{\widehat{\mathfrak{sl}_2}}(k,0)^G$ for $k\in\Z_{\geqslant 1}$ and some finite subgroup $G$ of Aut$(L_{\widehat{\mathfrak{sl}_2}}(k,0))$. 
    Representations and fusion rules of the $\mathbb{Z}_2$-orbifold of the vertex operator algebra $L_{\widehat{\frak{sl}_2}}(k,0) (k\in\mathbb{Z}_{\geqslant 1})$ were given in \cite{JW2}. 
    For the Klein group $K$, $k\in\mathbb{Z}_{\geqslant 1}$, representations of the orbifold vertex operator algebras $L_{\widehat{\frak{sl}_2}}(k,0)^{K}$ were constructed in \cite{JBW}.
    Let $\Z_3$ be the cyclic subgroup of $\Aut(L_{\widehat{\mathfrak{sl}_2}}(k,0))$ generated by $\sigma$ which is defined by $\sigma(h) = h, \sigma(e) = \frac{-1+\sqrt{-3}}{2} e, \sigma(f) = \frac{-1-\sqrt{-3}}{2} f$, where $\{ h, e, f \}$ is a standard Chevalley basis of $\mathfrak{sl}_2$ with Lie brackets $[h,e] = 2e, [h,f] = -2f, [e,f] = h$. Then any irreducible $L_{\widehat{\mathfrak{sl}_2}}(k,0)^{\Z_3}$-module occurs in an irreducible $\tau$-twisted $L_{\widehat{\mathfrak{sl}_2}}(k,0)$-module for some $\tau \in \Z_3$ \cite{DRX17}.
    In this paper, we classify and construct all the irreducible modules for the orbifold vertex operator algebras $L_{\widehat{\mathfrak{sl}_2}}(k,0)^{\Z_3}$ for $k\geqslant 1$. We construct $\tau$-twisted modules of $L_{\widehat{\mathfrak{sl}_2}}(k,0)$ for each $\tau \in \Z_3$, and give the decomposition of each irreducible $\tau$-twisted $L_{\widehat{\mathfrak{sl}_2}}(k,0)$-module into a direct sum of irreducible $L_{\widehat{\mathfrak{sl}_2}}(k,0)^{\Z_3}$-modules. It turns out that there are exactly $9(k+1)$ inequivalent irreducible $L_{\widehat{\mathfrak{sl}_2}}(k,0)^{\Z_3}$-modules. We call the irreducible $L_{\widehat{\mathfrak{sl}_2}}(k,0)^{\mathbb{Z}_3}$-module coming from the irreducible $L_{\widehat{\mathfrak{sl}_2}}(k,0)$-module the \emph{untwisted type} $L_{\widehat{\mathfrak{sl}_2}}(k,0)^{\mathbb{Z}_3}$-module. And, we call the irreducible $L_{\widehat{\mathfrak{sl}_2}}(k,0)^{\mathbb{Z}_3}$-module coming from the twisted $L_{\widehat{\mathfrak{sl}_2}}(k,0)$-module the \emph{twisted type} $L_{\widehat{\mathfrak{sl}_2}}(k,0)^{\mathbb{Z}_3}$-module.

    The quantum dimensions of the irreducible modules introduced in \cite{DJX13} are the important invariants of $V$ and the product formula $qdim_V (M \boxtimes_V N) = qdim_V M \cdot qdim_V N$ (\cite{DJX13}) for any $V$-modules $M$, $N$ plays an essential role in computing the fusion rules. An explicit relation between the quantum dimension of an irreducible $g$-twisted $V$-module M and the quantum dimension of an irreducible $V^G$-submodule of $M$ was given in \cite{DRX17}. We use this powerful relation to compute the quantum dimension of any irreducible module of the orbifold vertex operator algebras $L_{\widehat{\mathfrak{sl}_2}}(k,0)^{\Z_3}$. 
    
    The fusion rules for the orbifold vertex operator algebra $L_{\widehat{\mathfrak{sl}_2}}(k,0)^{\mathbb{Z}_3}$ are completely determined in Section 4.  
    The initial inspiration for the main idea is the fusion rules of the $\mathbb{Z}_2$-orbifold of the vertex operator algebra $L_{\widehat{\frak{sl}_2}}(k,0)$ \cite{JW2}, which is useful to determine the fusion products between untwisted type $L_{\widehat{\mathfrak{sl}_2}}(k,0)^{\mathbb{Z}_3}$-modules and untwisted type $L_{\widehat{\mathfrak{sl}_2}}(k,0)^{\mathbb{Z}_3}$-modules as well as the fusion products between untwisted type $L_{\widehat{\mathfrak{sl}_2}}(k,0)^{\mathbb{Z}_3}$-modules and twisted type $L_{\widehat{\mathfrak{sl}_2}}(k,0)^{\mathbb{Z}_3}$-modules.
    However, the determination of the fusion products between twisted type $L_{\widehat{\mathfrak{sl}_2}}(k,0)^{\mathbb{Z}_3}$-modules and twisted type $L_{\widehat{\mathfrak{sl}_2}}(k,0)^{\mathbb{Z}_3}$-modules is much more complicated. The main strategy is to employ the Proposition 2.8 in \cite{DLM96-1} which described that if $W = M_1 \boxtimes_V M_2$ for any $g_i$-twisted $V$-module $M_i(i=1,2)$ together with some other conditions then $\widetilde{W} = M_1 \boxtimes_V \widetilde{M_2}$ (the notation of $\widetilde{W}$ is defined in \cite{DLM96-1} Lemma 2.6).
    Furthermore, we determine the contragredient modules of all the irreducible $L_{\widehat{\mathfrak{sl}_2}}(k,0)^{\mathbb{Z}_3}$-modules, thus the fusion rules for $L_{\widehat{\mathfrak{sl}_2}}(k,0)^{\mathbb{Z}_3}$ are completely determined.
    % Finally, we use the lattice realization of $L(1,0)^{\Z_3}$ and the correspondence between irreducible $L(1,0)^{\Z_3}$-modules and \{$V_{\Z\be+\frac{s}{18}\be}|(\be, \be) = 18, 0\leqslant s<18$\} to determine the fusion rules for $L(1,0)^{\mathbb{Z}_3}$ which is compatible with our main results. 

    The paper is organized as follows. In Section 2, we briefly review some basic notations and facts on vertex operator algebras. In Section 3, we first give the action of the cyclic group $\Z_3$ on $L_{\widehat{\mathfrak{sl}_{2}}}(k,0)$ and realize each element of $\Z_3$ as an inner automorphism of $\mathfrak{sl}_2$. Then we classify and construct all the irreducible modules of the orbifold vertex operator algebras $L_{\widehat{\mathfrak{sl}_2}}(k,0)^{\Z_3}$ for $k\geqslant 1$. In Section 4, we compute the quantum dimension of any irreducible module of $L_{\widehat{\mathfrak{sl}_2}}(k,0)^{\Z_3}$ for $k\geqslant 1$. Finally, the fusion rules for the orbifold vertex operator algebras $L_{\widehat{\mathfrak{sl}_2}}(k,0)^{\mathbb{Z}_3}$ are completely determined.

    We use the usual symbols $\C$ for the complex numbers, $\Z$ for the integers, $\Z_{\geqslant 0}$ for the nonnegative integers, and $\Z_{\geqslant 1}$ for the positive integers. In this paper, $\overline{j}$ means the residue of the integer $j$ modulo $3$.

\section{Preliminary}

    Let $(V,Y,\mathbbm{1},\omega)$ be a vertex operator algebra \cite{B86}, \cite{FLM88}. We first review basics from \cite{DLM98-1}, \cite{DLM00}, \cite{FHL93} and \cite{LeL04}. Let $g$ be an automorphism of the vertex operator algebra $V$ of finite order $T$. Denote the decomposition of $V$ into eigenspaces of $g$ as:
        \[ V=\bigoplus_{r\in \mathbb{Z}/T\mathbb{Z}}V^r,\]
    where $V^r= \{v \in V |gv = e^{-2\pi \sqrt{-1}\frac{r}{T}}v\}$, $0 \leqslant r \leqslant T-1$. We use $r$ to denote both an integer between $0$ and $T-1$ and its residue class modulo $T$ in this situation.

\begin{defn}
    Let $V$ be a vertex operator algebra. A weak $g$-twisted $V$-module is a vector space $M$ equipped with a linear map
    \begin{align*}
        Y_M(\cdot, x) : V & \longrightarrow (\End M)[[x^{\frac{1}{T}},x^{-\frac{1}{T}}]] \\
                v & \longmapsto Y_M(v,x) = \sum_{n\in \frac{1}{T}\mathbb{Z}}v_nx^{-n-1},
    \end{align*}
    where $v_n \in \End M $, satisfying the following conditions for $0 \leqslant r \leqslant T-1$, $u \in V^r, v \in V$, $w \in M$:
    \[Y_M(u,x) = \sum_{n\in \frac{r}{T}+\mathbb{Z}}u_nx^{-n-1},\]
    \[u_sw=0 \quad\text{for} \quad s \gg 0,\]
    \[Y_M(\mathbbm{1},x) = id_M,\]
    \[x_{0}^{-1}\delta(\frac{x_1-x_2}{x_0})Y_M(u,x_1)Y_M(v,x_2)-x_{0}^{-1}\delta(\frac{x_2-x_1}{-x_0})Y_M(v,x_2)Y_M(u,x_1)\]
    \[=x_{2}^{-1}(\frac{x_1-x_0}{x_2})^{-\frac{r}{T}}\delta(\frac{x_1-x_0}{x_2})Y_M(Y(u,x_0)v,x_2),\]
    where $\delta(x)=\sum_{n\in\mathbb{Z}}x^n$ and all binomial expressions are to be expanded in nonnegative integral powers of the second variable.
\end{defn}

The following Borcherds identities can be derived from the twisted-Jacobi identity \cite{DLM98-1}, \cite{XX98}.

\begin{equation}
    [ u_{m + \frac{r}{T}}, v_{n + \frac{s}{T}} ] = \sum_{i=0}^{\infty} \binom{m + \frac{r}{T}}{i}(u_iv)_{m + n + \frac{r+s}{T}-i},     \label{Borcherds identity 1}
\end{equation}

\begin{equation}
    \sum_{i=0}^{\infty} \binom{\frac{r}{T}}{i}(u_{m+i}v)_{n+\frac{r+s}{T}-i} = \sum_{i=0}^{\infty} (-1)^i \binom{m}{i} (u_{m+\frac{r}{T}-i}v_{n+\frac{s}{T}+i} -(-1)^mv_{m+n+\frac{s}{T}-i}u_{\frac{r}{T}+i}),        \label{Borcherds identity 2}
\end{equation}
    where $u \in V^r$, $v \in V^s$, $m$, $n \in \mathbb{Z}$.

\begin{defn}
    An admissible $g$-twisted $V$-module is a weak $g$-twisted $V$-module which carries a $\frac{1}{T}\mathbb{Z}_{\geqslant 0}$-grading $M = \oplus_{n\in \frac{1}{T}\mathbb{Z}_{\geqslant 0}}M(n)$ satisfying $v_mM(n) \subseteq M(n+r-m-1)$ for homogeneous $v \in V_r$, $m$, $n \in\frac{1}{T}\mathbb{Z}$.
\end{defn}

\begin{defn}
    A $g$-twisted $V$-module is a weak $g$-twisted $V$-module which carries a $\mathbb{C}$-grading:
    \[M=\oplus_{ \lambda \in \mathbb{C}}M_\lambda,\]
    such that dim $M_{\lambda} < \infty$, $M_{\lambda +\frac{n}{T}} = 0$ for fixed $\lambda$ and $n \ll 0$, $L(0)w = \lambda w = ( \wt w )w$ for $w \in M_{\lambda}$, where $L(0)$ is the component operator of $Y_M(\omega, x) = \sum_{n\in \mathbb{Z}}L(n)x^{-n-2}$.
\end{defn}

\begin{rmk}
     If $g = id_V$, we have the notations of weak, admissible and ordinary $V$-modules \cite{DLM97-1}.
\end{rmk}

\begin{defn}
    A vertex operator algebra $V$ is called $g$-rational if the admissible $g$-twisted $V$-module category is semisimple. $V$ is called rational if $V$ is $id_V$-rational.
\end{defn}

If $M = \oplus_{n\in \frac{1}{T}\mathbb{Z}_{\geqslant 0}}M(n)$ is an admissible $g$-twisted $V$-module, the contragredient module $M'$ is defined as follows:
     \[ M' = \oplus_{n\in \frac{1}{T}\mathbb{Z}_{\geqslant 0}}M(n)^{*},\]
where $M(n)^{*} = \Hom_{\C}(M(n), \mathbb{C})$. The vertex operator $Y_{M'}(a, z)$ is defined for $a \in V$ via
     \[ \langle Y_{M'}(a, z)f, u \rangle = \langle f, Y_M(e^{zL(1)}(-z^{-2})^{L(0)}a, z^{-1})u \rangle, \]
where $\langle f, u \rangle = f(u)$ is the natural pairing $M' \times M \to \mathbb{C}$. It follows from \cite{FHL93} and \cite{XF00} that $( M' , Y_{M'} )$ is an admissible $g^{-1}$-twisted $V$-module. We can also define the contragredient module $M'$ for a $g$-twisted $V$-module $M$. In this case, $M'$ is a $g^{-1}$-twisted $V$-module. Moreover, $M$ is irreducible if and only if $M'$ is irreducible. $M$ is said to be selfdual if $M$ is $V$-isomorphic to $M'$. In particular, $V$ is said to be a selfdual vertex operator algebra if $V$ is isomorphic to $V'$.
We recall the following  concept from \cite{Zhu96}.

\begin{defn}
    A vertex operator algebra is called $C_2$-cofinite if $C_2(V)$ has finite codimension (i.e., dim $V/C_2(V) < \infty$),
    where $C_2(V) = \langle u_{-2}v \mid u, v \in V \rangle$.
\end{defn}

We have the following result from \cite{ABD04}, \cite{DLM98-1} and \cite{Zhu96}.

\begin{thm}
    If $V$ is a vertex operator algebra satisfying the $C_2$-cofinite property, then $V$ has only finitely many irreducible admissible modules up to isomorphism. The rationality of $V$ also implies the same result.
\end{thm}

We have the following results from \cite{DLM98-1} and \cite{DLM00}.

\begin{thm}
    If V is $g$-rational vertex operator algebra, then

        (1) Any irreducible admissible $g$-twisted $V$-module $M$ is a $g$-twisted $V$-module. Moreover, there exists a number $\lambda \in \mathbb{C}$ such that $M=\oplus_{n\in \frac{1}{T}\mathbb{Z}_{\geqslant 0}}M_{\lambda + n}$, where $M_{\lambda} \ne 0$. The number $\lambda$ is called the conformal weight of $M$;

        (2)There are only finitely many irreducible admissible $g$-twisted $V$-modules up to isomorphism.
\end{thm}

\begin{defn}
    A vertex operator algebra $V$ is called regular if every weak $V$-module is a direct sum of irreducible $V$-modules, i.e., the weak module category is semisimple.
\end{defn}

\begin{defn}
    A vertex operator algebra $V=\oplus_{n\in \mathbb{Z}}V_n$ is said to be of CFT type if $V_n = 0$ for $n < 0$ and $V_0 = \mathbb{C}\mathbbm{1}$.
\end{defn}

\begin{rmk}
    It is proved in \cite{ABD04}  that for a CFT type vertex operator algebra $V$,  regularity is equivalent to rationality and $C_2$-cofiniteness.
\end{rmk}

\begin{thm} (\cite{CM}, \cite{M15}) \label{VG prop}
If $V$ is a regular and selfdual vertex operator algebra of CFT type, and $G$ is solvable, then $V^G$ is a regular and selfdual vertex operator algebra of CFT type.
\end{thm}

We now review some notations and  facts about the action of the automorphism group on twisted modules of vertex operator algebra $V$ from \cite{DLM00}, \cite{DRX17}, \cite{DY02}, \cite{MT04}.

  Let $g, h$ be two automorphisms of $V$. If $(M, Y_M)$ is a weak $g$-twisted $V$-module, there is a weak $h^{-1}gh$-twisted $V$-module $(M \circ h, Y_{M \circ h})$ where $M \circ h = M$ as vector spaces and $Y_{M \circ h}(v, z) = Y_M(hv, z)$ for $v \in V$. This gives a right action of $\Aut(V)$ on weak twisted $V$-modules. Symbolically, we write
    \[  (M, Y_M) \circ h = (M \circ h,Y_{M \circ h}) = M \circ h. \]
  The $V$-module $M$ is called \emph{$h$-stable} if $M \circ h$ and $M$ are isomorphic $V$-modules.

  Let $G$ be a finite group of automorphisms of $V$, $g \in G$ of finite order $T$ and $M = (M, Y_M)$ an irreducible $g$-twisted $V$-module. Define a subgroup $G_M$ of $G$ consisting all of $h \in G$ such that $M$ is $h$-stable. For $h \in G_{M}$, there is a linear isomorphism $\phi(h) : M \to M$ satisfying
    \[ \phi(h)Y_M(v, z){\phi(h)}^{-1} = Y_{M \circ h}(v, z) = Y_M(hv, z)  \]
  for $v \in V$. The simplicity of $M$ together with Schur's lemma shows that $h \mapsto \phi(h)$ is a projective representation of $G_M$ on $M$. Let $\alpha_M$ be the corresponding $2$-cocycle in $C^2(G, \mathbb{C}^*)$. Then $M$ is a module for the twisted group algebra $\mathbb{C}^{\alpha_M}[G_M]$ which is  a semisimple associative algebra. A basic fact is that $g$ belongs to $G_M$. Let $M^r=\oplus_{n\in \frac{r}{T}+\mathbb{Z}_{\geqslant 0}}M(n)$ for $r = 0$, $1$, $\cdots$, $T-1$, then $M = \oplus_{n\in \frac{1}{T}\mathbb{Z}_{\geqslant 0}}M(n) = \oplus_{r=0}^{T-1}M^r$ and each $M^r$ is an irreducible $V^{\langle g \rangle}$-module on which  $\phi(g)$ acts as constant $e^{2\pi \sqrt{-1}\frac{r}{T}}$ \cite{DM97}, \cite{DRX17}.

Let $\Lambda_{G_M, \alpha_M}$ be the set of all irreducible characters $\lambda$ of $\mathbb{C}^{\alpha_M}[G_M]$. Then
    \begin{equation} \label{any irr twist mod decomp}
       M = \oplus_{\lambda \in \Lambda_{G_M, \alpha_M}}W_{\lambda} \otimes M_{\lambda},
    \end{equation}
    where $W_{\lambda}$ is the simple $\mathbb{C}^{\alpha_M}[G_M]$-module affording $\lambda$ and $M_{\lambda} = \Hom_{\mathbb{C}^{\alpha_M}[G_M]}(W_{\lambda}, M)$ is the mulitiplicity of $W_{\lambda}$ in $M$. And each $M_{\lambda}$ is a module for the vertex operator subalgebra $V^{G_M}$.

The following results follow from \cite{DRX17} and \cite{DY02}.

\begin{thm}     \label{DRX17 thm1}
    With the same notations as above we have

    (1) $W_{\lambda} \otimes M_{\lambda}$ is nonzero for any $\lambda \in \Lambda_{G_M, \alpha_M}$.

    (2) Each $M_{\lambda}$ is an irreducible $V^{G_M}$-module.

    (3) $M_{\lambda}$ and $M_{\mu}$ are equivalent $V^{G_M}$-module if and only if $\lambda = \mu$.
\end{thm}

\begin{thm}     \label{DRX17 thm2}
    Let $g$, $h \in G$, $M$ be an irreducible $g$-twisted $V$-module, and $N$ an irreducible $h$-twisted $V$-module. If $M$, $N$ are not in the same orbit under the action of $G$, then the irreducible $V^G$-modules $M_{\lambda}$ and $N_{\mu}$ are inequivalent for any $\lambda \in \Lambda_{G_M, \alpha_M}$ and $\mu \in \Lambda_{G_N, \alpha_N}$.
\end{thm}

\begin{thm}   \label{DRX17 thm3}
    Let $V^G$ be a regular and selfdual vertex operator algebra of CFT type. Then any irreducible $V^G$-module is isomorphic to $M_\lambda$ for some irreducible $g$-twisted $V$-module $M$ and some $\lambda \in \Lambda_{G_M, \alpha_M}$. In particular, if $V$ is a regular and selfdual vertex operator algebra of CFT type and $G$ is solvable, then any irreducible $V^G$-module is isomorphic to some $M_\lambda$.
\end{thm}

We now recall from \cite{FHL93} the notions of intertwining operators and fusion rules.

\begin{defn} 
     Let $(V, Y)$ be a vertex operator algebra and let $(W^{1}, Y^{1})$, $(W^{2}, Y^{2})$ and $(W^{3}, Y^{3})$ be $V$-modules. An \emph{intertwining operator} of type 
     $\left(\begin{array}{c}
     W^{3}\\
     W^{1} \ W^{2}
     \end{array}\right)$ is a linear map
     \begin{align*}
     I(\cdot, z) : W^{1} & \longrightarrow \Hom ( W^{2}, W^{3} ) \{ z \} \\
                 u  & \longmapsto I(u, z) = \sum_{n\in\mathbb{Q}}u_{n}z^{-n-1}
     \end{align*}
     satisfying:
        
        (1) for any $u\in W^{1}$ and $v\in W^{2}$, $u_{n}v=0$ for $n$
        sufficiently large;
        
        (2) $I(L(-1)v, z)=\frac{d}{dz}I(v, z)$;
        
        (3) (Jacobi identity) for any $u\in V$, $v\in W^{1}$,

        \[ z_{0}^{-1}\delta\left(\frac{z_{1}-z_{2}}{z_{0}}\right)Y^{3}(u, z_{1})I(v, z_{2})-z_{0}^{-1}\delta\left(\frac{-z_{2}+z_{1}}{z_{0}}\right)I(v, z_{2})Y^{2}(u, z_{1})  \]
        \[ =z_{2}^{-1}\delta\left(\frac{z_{1}-z_{0}}{z_{2}}\right)I(Y^{1}(u, z_{0})v, z_{2}).  \]
    The space of all intertwining operators of type $\left(\begin{array}{c}
         W^{3}\\
         W^{1}\ W^{2}
         \end{array}\right)$ is denoted by
         $I_{V}\left(\begin{array}{c}
         W^{3}\\
         W^{1}\ W^{2}
         \end{array}\right).$ Let $N_{W^{1}, W^{2}}^{W^{3}}=\dim I_{V}\left(\begin{array}{c}
         W^{3}\\
         W^{1}\ W^{2}
         \end{array}\right)$. These integers $N_{W^{1}, W^{2}}^{W^{3}}$ are usually called the
         \emph{fusion rules}.
\end{defn}

\begin{rmk}\label{Intertwining expression} (\cite{FZ92}) 
    Let $M^{i}=\oplus_{n\in\mathbb{Z}}M^{i}(n)$, $i=1, 2, 3$ be irreducible modules for a vertex operator algebra $V$, and the corresponding conformal weights are $a_{i}$, $i=1, 2, 3$. If $I(\cdot,z)$ is an intertwining operator of type $\left(\begin{array}{c} 
           W^{3} \\
         W^{1} \ W^{2}
    \end{array}\right)$,
    then $I(\cdot,z)$ can be written as
        \[   I(v, z) = \sum_{n\in\mathbb{Z}}v(n)z^{-n-1}z^{-a_1-a_2+a_3}  \]
    such that for honogeneous $v\in M^{1},$ $v(n)M^{2}(m)\subset M^{3}(m+deg v-1-n)$, where $deg v = d$ means $v\in M^{1}(d).$
\end{rmk}

% We will write $o(v)=v(deg\, v-1).$

% \begin{defn} 
%     Let $V$ be a vertex operator algebra, and $W^{1}$, $W^{2}$ be two $V$-modules. A module $(W,I)$, where $I\in I_{V}\left(\begin{array}{c}
%     \ \ W\ \\
%     W^{1}\ \ W^{2}
%     \end{array}\right),$ is called a \emph{tensor product} (or fusion product) of $W^{1}$
%     and $W^{2}$ if for any $V$-module $M$ and $\mathcal{Y}\in I_{V}\left(\begin{array}{c}
%     \ \ M\ \\
%     W^{1}\ \ W^{2}
%     \end{array}\right),$ there is a unique $V$-module homomorphism $f:W\rightarrow M,$ such
%     that $\mathcal{Y}=f\circ I.$ As usual, we denote $(W,I)$ by $W^{1}\boxtimes_{V}W^{2}$ or $W^{1}\boxtimes W^{2}$ simply.
% \end{defn}

% \begin{rmk}
%      It is well known that if $V$ is rational, then for any two irreducible
%      $V$-modules $W^{1}$ and $W^{2},$ the fusion product $W^{1}\boxtimes_{V}W^{2}$ exists and
%      $$
%      W^{1}\boxtimes_{V}W^{2}=\sum_{W}N_{W^{1},\ W^{2}}^{W}W,
%      $$
%      where $W$ runs over the set of equivalence classes of irreducible
%      $V$-modules.
% \end{rmk}

From \cite{ADL05}, we have the following proposition.  

\begin{prop}
    Let $V$ be a vertex operator algebra and let $W^1$, $W^2$, $W^3$ be $V$-modules among which $W^1$ and $W^2$ are irreducible. Suppose that $U$ is a vertex operator subalgebra of $V$ (with the same Virasoro element) and that $N^1$ and $N^2$ are irreducible $U$-submodules of $W^1$ and $W^2$, respectively. Then the restriction map from
        $I_{V}\left(
            \begin{array}{c}
                 \ W^{3}\ \\
                W^{1} \ W^{2}
            \end{array}\right)$ to 
        $I_{U}\left( 
        \begin{array}{c}
             \ W^{3}\ \\
            N^{1} \ N^{2}
        \end{array}\right)$ is injective. In particular,
        \begin{equation} 
            dim I_{V}\left(
            \begin{array}{c}
                 \ W^{3}\ \\
                W^{1} \ W^{2}
            \end{array}\right) 
            \leqslant dim
            I_{U}\left( 
            \begin{array}{c}
                 \ W^{3}\ \\
                N^{1} \ N^{2}
            \end{array}\right)  
        \end{equation}
\end{prop}

\begin{defn}
    Let $V$ be a vertex operator algebra, and $W^1$ , $W^2$ be two $V$-modules. A pair $(W, F(\cdot, z))$, which consists of a $V$-module $W$ and an intertwining operator $F(\cdot, z)$ of type 
         $\left( 
        \begin{array}{c}
             \  W \ \\
            W^{1} \ W^{2}
        \end{array}\right)$, 
    is called a tensor product (or fusion product) of the ordered pair $W^1$ and $W^2$ if for any $V$-module $M$ and any intertwining operator $I(\cdot, z)$ of type
        $\left( 
        \begin{array}{c}
             \  M \ \\
            W^{1} \ W^{2}
        \end{array}\right)$, 
    there exists a unique $V$-module homomorphism $f$ from $W$ to $M$ such that $I(\cdot, z) = f \circ F(\cdot, z)$. In this case, we denote the tensor product $(W, F(\cdot, z))$ by $W^1 \boxtimes_VW^2$.
\end{defn}

The following result is obtained in \cite{H95}, \cite{HL95-1}, \cite{HL95-2}.

\begin{thm}
    Let $V$ be a regular and selfdual vertex operator algebra of CFT type, $M^0 \cong V, M^1, \cdots, M^d$ are all inequivalent irreducible $V$-modules and the conformal weights $\lambda_i$ of $M^i$ are positive for all $i > 0$. Then the tensor product of any two $V$-modules $M \boxtimes_VN$ exists. In particular,
        \begin{equation}
            M^i \boxtimes_V M^j = \sum_{k=0}^d N_{M^i, M^j}^{M^k}M^k,
        \end{equation}
    for any $i, j \in \{0, 1, \cdots, d\}$.
\end{thm}

Fusion rules have the following symmetric property \cite{FHL93}.

\begin{prop}   \label{fusionsymm.}
     Let $W^{i} (i=1,2,3)$ be $V$-modules. Then
     $$N_{W^{1},W^{2}}^{W^{3}}=N_{W^{2},W^{1}}^{W^{3}}, \ N_{W^{1},W^{2}}^{W^{3}}=N_{W^{1},(W^{3})^{'}}^{(W^{2})^{'}}.$$
\end{prop}
\begin{defn} 
     Let $V$ be a simple vertex operator algebra, a simple $V$-module $M$ is called a simple current if for any irreducible $V$-module, $M\boxtimes_V W$ exists and is also an irreducible $V$-module.
\end{defn}

\section{Classification and construction of irreducible modules of $L_{\widehat{\mathfrak{sl}_2}}(k,0)^{\Z_3}$}

In this section, we will introduce the cyclic group $\Z_3$ which is a subgroup of $\Aut(L_{\widehat{\frak{sl}_{2}}}(k,0))$, and realize each element of ${\Z_3}$ as an inner automorphism of $\frak{sl}_2(\C)$.
And we will classify and construct explicitly the irreducible modules of the orbifold vertex operator algebras $L_{\widehat{\mathfrak{sl}_2}}(k,0)^{\Z_3}$ for $k\geqslant 1$.

Let $h,e,f$ be a standard Chevalley basis of $\frak{sl}_2(\C)$, define automorphism $\sigma$ of $\frak{sl}_2(\C)$ as follows:
$$
\sigma(h) = h, \  \sigma(e) = \frac{-1+\sqrt{-3}}{2} e, \ \sigma(f) = \frac{-1-\sqrt{-3}}{2} f.
$$
It is obvious that the automorphic subgroup generated by $\sigma$ is isomorphic to the cyclic group ${\Z_3}$, and ${\Z_3}$ can be lifted to an automorphic subgroup of the vertex operator algebra $L_{\widehat{\mathfrak{sl}_{2}}}(k,0)$.

In the following statement, we denote $L_{\widehat{\mathfrak{sl}_2}}(k,0)$ by $L(k,0)$ for simplicity and $k$ is a positive integer unless otherwise stated. 
By the quantum Galois theory \cite{DM97}, we first have the following decomposition.

\begin{thm}   \label{V decomposition}
   As a $L(k,0)^{\Z_3}$-module,
   \[ L(k,0)= L(k,0)^0 \oplus L(k,0)^1 \oplus L(k,0)^2 ,\]
   where $L(k,0)^0 (= L(k,0)^{\Z_3})$ is a simple vertex operator algebra, and $L(k,0)^0$ $($resp. $L(k,0)^1$, $L(k,0)^2$$)$ is the irreducible $L(k,0)^{\Z_3}$-module generated by the lowest weight vector $\mathbbm{1}$ $($resp. $e(-1)\1$, $f(-1)\1$$)$ with the lowest weight $0$ $($resp. $1$, $1$$)$.
\end{thm}
\begin{proof}
   Since ${\Z_3}$ is a cyclic group which has only three 1-dimensional irreducible modules. Let Irr$(\Z_3)$ denote the set of irreducible characters of ${\Z_3}$ which  contains three irreducible characters $\chi_0$ (unit representation), $\chi_1$, and $\chi_2$ up to isomorphism. From \cite{DM97}, ${L(k,0)} = \oplus_{\chi \in \Irr(\Z_3)}{L(k,0)}_{\chi}$ is a decomposition of $L(k,0)$ into simple $L(k,0)^{\Z_3}$-modules. Moreover, ${L(k,0)}_{\chi}$ is nonzero for any $\chi \in \Irr(\Z_3)$, and $L(k,0)_{\chi}$ and $L(k,0)_{\mu}$ are equivalent $L(k,0)^{\Z_3}$-module if and only if $\chi=\mu$. Obviously, $L(k,0)^{\Z_3}$ is an irreducible $L(k,0)^{\Z_3}$-module affording the unit character $\chi_0$. Observing the action of ${\Z_3}$ on $L(k,0)_1$ which is isomorphic to $\mathfrak{sl}_2(\C)$, we find that $e(-1)\1$ and $f(-1)\1$ generate two inequivalent irreducible modules according to $\chi_1$ and $\chi_2$, respectively. Note that $L(k,0)_0 = \mathbb{C}\mathbbm{1}$ and $L(k,0)_1 = \C h(-1)\1 \oplus \mathbb{C}(e(-1)\1) \oplus \C(f(-1)\1)$, then $\mathbbm{1}$, $e(-1)\1$ and $f(-1)\1$ are three different lowest weight vectors in $L(k,0)$ as a $L(k,0)^{\Z_3}$-module. Let $L(k,0)^0$ $($resp. $L(k,0)^1, L(k,0)^2$$)$ be the irreducible $L(k,0)^{\Z_3}$-module generated by the lowest weight vector $\mathbbm{1}$ $($resp. $e(-1)\1$, $f(-1)\1$$)$ with the lowest weight $0$ $($resp. $1$, $1$$)$. Then the irreducible $L(k,0)^{\Z_3}$-module decomposition $L(k,0)= \oplus_{j=0}^{2}L(k,0)^j$ holds.
\end{proof}

Let $\alpha$ be the simple root of $\mathfrak{sl}_2(\C)$ with $\langle \alpha, \alpha  \rangle = 2$. From \cite{FZ92}, the integrable highest weight $L(k,0)$-modules $L(k, i)$ for $0 \leqslant i \leqslant k$ provide a complete list of irreducible $L(k,0)$-modules with the lowest weight spaces being  $(i+1)$-dimensional irreducible $\mathfrak{sl}_2(\C)$-modules $L(\frac{i\alpha}{2})$, respectively. For $0\leqslant i\leqslant k$, let $\{v^{i,j}|0 \leqslant j \leqslant i\}$ be the basis of $L(\frac{i\alpha}{2})$ according to the $\mathfrak{sl}_2$-triple $\{ h, e, f \}$ with the following action of $\widehat{\mathfrak{sl}_2}$ on $L(\frac{i\alpha}{2})$, namely
\[ h(0)v^{i,j} = (i - 2j) v^{i,j} \quad \text{for} \quad 0 \leqslant j \leqslant i , \]
\[ e(0)v^{i,0} = 0, \quad e(0)v^{i,j} = (i - j + 1)v^{i,j-1} \quad \text{for} \quad 1 \leqslant j \leqslant i , \]
\[ f(0)v^{i,i} = 0, \quad f(0)v^{i,j} = (j + 1)v^{i,j+1} \quad \text{for} \quad 0 \leqslant j \leqslant i-1 , \]
\[ a(n)v^{i,j} = 0 \quad \text{for} \quad a \in \{ h, e, f \}, \quad n \geqslant 1 .\]

The following lemma will be very useful later.
\begin{lem}  \label{two 0 in L(k,k)}
    $e(-1)v^{k,0} = 0$, $f(-1)v^{k,k} = 0$ in $L(k,k)$.
\end{lem} 
\begin{proof}
    Since \{$h(-1)\1$, $e(-1)\1$, $f(-1)\1$\} is a gererator set of the simple vertex operator algebra $L(k,0)$. And
         \[ (h(-1)\1)_1e(-1)v^{k,0}=(e(-1)\1)_1e(-1)v^{k,0}=(f(-1)\1)_1e(-1)v^{k,0}=0\] 
    implies that $e(-1)v^{k,0}$ is a lowestest weight vector in the irreducible $L(k,0)$-module $L(k,k)$, yielding a contradiction. Thus $e(-1)v^{k,0} = 0$ in $L(k,k)$. Similarly, we can prove that $f(-1)v^{k,k} = 0$ in $L(k,k)$.
\end{proof}

It is well known that $L(k,0)$ is a regular and selfdual vertex operator algebra of CFT type for $k\in\Z_{\geqslant 1}$ \cite{FZ92}, \cite{LeL04}. From Theorem \ref{VG prop}, $L(k,0)^{\Z_3}$ is again a regular and selfdual vertex operator algebra of CFT type. Thus, from Theorem \ref{DRX17 thm3}, any irreducible $L(k,0)^{\Z_3}$-module occurs in an irreducible $\tau$-twisted $L(k,0)$-module for some $\tau \in \Z_3 = \{ \si^0=id, \si, \si^2 \}$.

Now we are in a position to classify and construct all the irreducible $L(k,0)^{\Z_3}$-modules coming from the irreducible untwisted(i.e., $id$-twisted) $L(k,0)$-modules $L(k,i)$ $(0 \leqslant i \leqslant k)$. 
We first determine the subgroup ${(\Z_3)}_{L(k,i)}$ of ${\Z_3}$ which contains $\tau \in \Z_3$ such that $L(k,i)$ is $\tau$-stable.

\begin{lem}  \label{irr Z3_M subgroup}
   ${(\Z_3)}_{L(k,i)} = \Z_3$ for any $0 \leqslant i \leqslant k$.
\end{lem}
\begin{proof}
   For any $\tau \in \Z_3$, by the definition of $L(k,i) \circ \tau$, $L(k,i)$ and $L(k,i) \circ \tau$ have the same lowest weight. Observe that the lowest weights $\frac{i(i+2)}{4(k+2)}(0 \leqslant i \leqslant k)$ are pairwise different which implies that all the irreducible $L(k, 0)$-modules $L(k,i)(0 \leqslant i \leqslant k)$ are $\tau$-stable. Thus, ${(\Z_3)}_{L(k,i)} = \Z_3$.
\end{proof}

From (\ref{any irr twist mod decomp}), we know that $L(k,i)(0 \leqslant i \leqslant k)$ can be decomposed as $L(k,0)^{\Z_3}$-modules, and the case of $i = 0$ has been stated in Theorem \ref{V decomposition}. For $0 < i \leqslant k$, we define $\phi(\sigma^r)$ $(r=0, 1, 2)$ from $L(k,i)$ to $L(k,i)$ as follows:
    \begin{align}
        \phi(\sigma^0): v^{i,j} & \mapsto v^{i,j},  \label{phi(sigma 0)} \\
        \phi(\sigma^1): v^{i,j} & \mapsto \xi^{i-j}v^{i,j},       \label{phi(sigma 1)} \\
        \phi(\sigma^2): v^{i,j} & \mapsto \xi^{j-i}v^{i,j},  \label{phi(sigma 2)}
    \end{align}
where $\xi=\frac{-1+\sqrt{-3}}{2}$. It is easy to verify that $\phi(\sigma^r)(r=0, 1, 2)$ are $L(k,0)$-module isomorphisms. Using Theorem \ref{DRX17 thm1} and Theorem \ref{DRX17 thm2}, we have the following result.

\begin{thm} \label{irr decomposition}
    For each $0 < i \leqslant k$, we have the following irreducible $L(k,0)^{\Z_3}$-module decomposition.
    \begin{enumerate}
      \item
        If $k = 1$, $i = 1$, then
        \begin{equation}
         L(k,i) = L(1,1) = L(1,1)^0 \oplus L(1,1)^1 \oplus L(1,1)^2,
        \end{equation}
        where $L(1,1)^0$ $($resp. $L(1,1)^1$, $L(1,1)^2$$)$ is the irreducible $L(1,0)^{\Z_3}$-module generated by the lowest weight vector $v^{1,1}$ $($resp. $v^{1,0}$, $f(-2)v^{1,1}$$)$ with the lowest weight $\frac{1}{4}$ $($resp. $\frac{1}{4}$, $\frac{9}{4}$$)$.
      \item
        If $k > 1$, $i = 1$, then
        \begin{equation}
         L(k,i) = L(k,1) = L(k,1)^0 \oplus L(k,1)^1 \oplus L(k,1)^2,
        \end{equation}
        where $L(k,1)^0$ $($resp. $L(k,1)^1$, $L(k,1)^2$$)$ is the irreducible $L(k,0)^{\Z_3}$-module generated by the lowest weight vector $v^{1,1}$ $($resp. $v^{1,0}$, $f(-1)v^{1,1}$$)$ with the lowest weight $\frac{3}{4(k+2)}$ $($resp. $\frac{3}{4(k+2)}$, $\frac{4k+11}{4(k+2)}$$)$.
      \item
        If $1 < i \leqslant k$, then
        \begin{equation}
        L(k,i) = L(k,i)^0 \oplus L(k,i)^1 \oplus L(k,i)^2,
        \end{equation}
        where $L(k,i)^0$, $L(k,i)^1$, and $L(k,i)^2$ are the irreducible $L(k,0)^{\Z_3}$-modules generated by the lowest weight vectors $v^{i,i}$,  $v^{i,i-1}$ and $v^{i,i-2}$ with the same lowest weight $\frac{i(i+2)}{4(k+2)}$, respectively.
    \end{enumerate}
\end{thm}
\begin{proof}
    The simplicity of $L(k,i)$ shows that $\tau \mapsto \phi(\tau)$ gives a projective representation of ${\Z_3}$ on $L(k,i)$. By Lemma \ref{irr Z3_M subgroup}, the ${\Z_3}$-orbit $L(k,i) \circ \Z_3$ of $L(k,i)$ only contains itself. Let $\alpha_{L(k,i)}$ be the corresponding $2$-cocycle in $C^2(\Z_3, \mathbb{C}^*)$. Then $L(k,i)$ is a module for the twisted group algebra $\mathbb{C}^{\alpha_{L(k,i)}}[\Z_3]$ with relation
    $\phi(\sigma)\phi(\sigma) = \phi(\sigma^2)$. 
    The twisted group algebra $\mathbb{C}^{\alpha_{L(k,i)}}[\Z_3]$ is a commutative semisimple associative algebra which has three irreducible modules of dimension one.
    Let $L(k,i) = \oplus_{j=0}^{j=2}L(k,i)^j$ be the eigenspace decomposition, where $L(k,i)^j$ is the eigenspace for $\phi(\sigma)$ on $L(k,i)$ with eigenvalue $e^{\frac{2\pi \sqrt{-1}j}{3}}$.
    From the definition of $\phi(\si)$, we know that $v^{i,i}$ is a lowest weight vector of $L(k,i)^0$ and $v^{i,i-1}$ is a lowest weight vector of $L(k,i)^1$. However, the lowest weight vectors of $L(k,0)^2$ depend on the value of $k$ and $i$.
    From lemma \ref{two 0 in L(k,k)}, we know that $e(-1)v^{1,0}=f(-1)v^{1,1} = 0$ in $L(1,1)$. Thus $f(-2)v^{1,1}$ is a lowest weight vector of $L(k,0)^2$ if $k=1$ and $f(-1)v^{1,1}$ is a lowest weight vector of $L(k,0)^2$ if $k>1$.
    From Theorem \ref{DRX17 thm1},
    we know that $L(k,i)^j$, $j = 0, 1, 2$ are inequivalent irreducible $L(k,0)^{\Z_3}$-modules for fixed $1 \leqslant i\leqslant k$.
    Therefore, the decomposition of $L(k,i)$ into inequivalent irreducible $L(k,0)^{\Z_3}$-modules is $L(k,i) = \oplus_{j=0}^{j=2}L(k,i)^j$.
\end{proof}

Let $h^{(r)}=\frac{r}{6}h$, $r \in \Z_{\geqslant 0}$. Direct calculations yield that
    \[ L(n)h^{(r)} = \delta_{n,0}h^{(r)}, \quad h^{(r)}(n)h^{(r)} = \delta_{n,1}\frac{r^2k}{18}\1,\quad \text{for} \quad n \in \mathbb{Z}_{\geqslant 1}, \]
    \[ h^{(r)}(0)e = \frac{r}{3}e, \quad h^{(r)}(0)f = -\frac{r}{3}f, \quad h^{(r)}(0)h^{(r)} = 0,\]   %20191112
    where $L(n)=\omega(n+1)$, $\omega$ is the conformal vector of $L(k,0)$. These equations show that $h^{(r)}(0)$ acts on $L(k,0)$ semisimply with rational eigenvalues. From \cite{Li96-2}, we know that $e^{2\pi \sqrt{-1}h^{(r)}(0)}$ is an automorphism of $L(k,0)$. Moreover, $e^{2\pi \sqrt{-1}h^{(r)}(0)}(h) = h$, $e^{2\pi \sqrt{-1}h^{(r)}(0)}(e) = e^{\frac{2\pi \sqrt{-1}r}{3}}e$, $e^{2\pi \sqrt{-1}h^{(r)}(0)}(f) = e^{-\frac{2\pi \sqrt{-1}r}{3}}f$. Thus we have the following proposition.

\begin{prop}
    $e^{2\pi \sqrt{-1}h^{(1)}(0)} = \sigma$, $e^{2\pi \sqrt{-1}h^{(2)}(0)} = \sigma^2$.
\end{prop}

For $r \in \Z_{\geqslant 0}$, let
\[\Delta(h^{(r)}, z) = z^{h^{(r)}(0)}\exp(\sum^{\infty}_{n=1}\frac{h^{(r)}(n)}{-n}(-z)^{-n}).\]
It is easy to verify that $\Delta(h^{(r)}, z) = \Delta(h^{(1)}, z)^r$. 
From \cite{Li97-2}, we have the following result.

\begin{lem}
    For each $r \in \Z_{\geqslant 0}$, $( L(k,i)^{T_r}, Y_{\sigma^r}(\cdot, z)) = ( L(k,i), Y(\Delta(h^{(r)}, z)\cdot, z) ) ( 0 \leqslant i \leqslant k )$ provide a complete list of irreducible $\sigma^r$-twisted $L(k, 0)$-modules. In particular, $( L(k,i)^{T_0}, Y_{\sigma^0}(\cdot, z)) = ( L(k,i), Y(\cdot, z) ) ( 0 \leqslant i \leqslant k )$ are all the irreducible untwisted $L(k, 0)$-modules.  
\end{lem}

Direct calculations yield that
    \begin{equation}
        h^{(r)}(0)\omega = 0, \quad h^{(r)}(1)\omega = h^{(r)}, \quad h^{(r)}(1)^{2}\omega = \frac{r^2k}{18}\mathbbm{1},            \label{twist-h1}
    \end{equation}
    \begin{equation}
        h^{(r)}(n)\omega = 0  \ \ \  \text{for} \ \ \  n \in \mathbb{Z}_{>1},                          \label{twist-h2}
    \end{equation}
    \begin{equation}
        \Delta(h^{(r)}, z)\omega = \omega + z^{-1}h^{(r)} + z^{-2}\frac{r^2k}{36}\mathbbm{1},  \label{twist-h3}
    \end{equation}
    \begin{equation}
        Y_{\sigma^r}(h^{(r)}, z) = Y(h^{(r)}+\frac{r^2k}{18}z^{-1}, z),   \label{twist-h4}
    \end{equation}
    \begin{equation}
        Y_{\sigma^r}(h, z) = Y(h+\frac{rk}{3}z^{-1}, z),    \label{twist-h5}
    \end{equation}
    \begin{equation}
        Y_{\sigma^r}(e, z) =z^{\frac{r}{3}}Y(e, z),   \label{twist-e}
    \end{equation}
    \begin{equation}
        Y_{\sigma^r}(f, z) =z^{-\frac{r}{3}}Y(f, z).     \label{twist-f}
    \end{equation}
To distinguish the components of $Y(v,z)$ from those of $Y_{\sigma^r}(v,z)$, for fixed $r$, we denote the following expansions
    \[ Y_{\sigma^r}(v,z) = \sum_{n\in \frac{t}{3}+\mathbb{Z}}v_nz^{-n-1}, \quad  Y(v,z) = \sum_{n\in\mathbb{Z}}v(n)z^{-n-1}, \]
    where $v \in L(k, 0)$, $t \in \{0, 1, 2\}$ such that $\sigma^r(v) = e^{\frac{-2\pi \sqrt{-1}t}{3}}v$. 
    And we denote $L_n^{(r)}$ be the component operator of $Y_{\sigma^r}(\omega, z) = \sum_{n\in \mathbb{Z}}L_n^{(r)}z^{-n-2}$. Note that $L_n^{(0)} = L(n)$.  
    By (\ref{twist-h3})-(\ref{twist-f}), Lemma  \ref{two 0 in L(k,k)} and direct calculations, we have the following lemmas.

\begin{lem}  \label{lw a_kir}
    $L_0^{(r)}v^{i,i} = a_{k,i}^{(r)}v^{i,i}$, where $a_{k,i}^{(r)} = \frac{i(i+2)}{4(k+2)} + \frac{r^2k-6ir}{36}$ is the eigenvalue of the operator $L_0^{(r)}$ on $v^{i,i}$. Thus, for $r = 0, 1, 2$, $a_{k,i}^{(r)}$ is the conformal weight of the irreducible $\sigma^r$-twisted $L(k, 0)$-module $L(k,i)^{T_r}$. 
\end{lem}

\begin{lem}
    For $0\leqslant i\leqslant k$, write $L(k,i)^{T_1} = \oplus_{n \in \frac{1}{3}\Z_{\geqslant 0}}L(k,i)^{T_1}(n)$ as an admissible $\si$-twisted $L(k,0)$-module. Then 
    \begin{enumerate}
        \item   For  $i=0$, 
            \[L(k,i)^{T_1}(0)=\C \1, \ \ \ \  L(k,i)^{T_1}(\frac{1}{3}) = \C e_{-\frac{1}{3}}\1=0, \]
            \[L(k,i)^{T_1}(\frac{2}{3}) = \C f_{-\frac{2}{3}}\1=\C f(-1)\1, \ \ \ \ e_{-\frac{4}{3}}\1 = e(-1)\1 \in L(k,i)^{T_1}(\frac{4}{3}),     \]
            \[ L_0^{(1)}(\1)=\frac{k}{36}\1, \ \ \ \  L_0^{(1)}(e(-1)\1)=(\frac{k}{36}+\frac{4}{3})e(-1)\1,  \]
            \[ L_0^{(1)}(f(-1)\1)=(\frac{k}{36}+\frac{2}{3})f(-1)\1.   \]
        \item   For  $i=1$ and $k=1$, 
            \[L(k,i)^{T_1}(0)=\C v^{1,1} , \ \ \ \  L(k,i)^{T_1}(\frac{1}{3}) = \C e_{-\frac{1}{3}}v^{1,1}=\C v^{1,0}, \]
            \[L(k,i)^{T_1}(\frac{2}{3}) = \C e_{-\frac{1}{3}}^2v^{1,1} \oplus \C f_{-\frac{2}{3}}v^{1,1}=0,  \]
            \[ f_{-\frac{5}{3}}v^{1,1} = f(-2)v^{1,1} \in L(k,i)^{T_1}(\frac{5}{3}),    \]
            \[ L_0^{(1)}(v^{1,1})=\frac{1}{9}v^{1,1}, \ \ \ \  L_0^{(1)}(v^{1,0})=\frac{4}{9}v^{1,0}, \]
            \[ L_0^{(1)}(f(-2)v^{1,1})=\frac{16}{9}f(-2)v^{1,1}.   \]
        \item   For  $i=1$ and $k>1$, 
            \[ L(k,i)^{T_1}(0)=\C v^{1,1} , \ \ \ \  L(k,i)^{T_1}(\frac{1}{3}) = \C e_{-\frac{1}{3}}v^{1,1}=\C v^{1,0}, \]
            \[ L(k,i)^{T_1}(\frac{2}{3}) = \C e_{-\frac{1}{3}}^2v^{1,1} \oplus \C f_{-\frac{2}{3}}v^{1,1}=\C f(-1)v^{1,1},\]
            \[ L_0^{(1)}(v^{1,1})=(\frac{3}{4(k+2)}+\frac{k-6}{36})v^{1,1}, \]
            \[ L_0^{(1)}(v^{1,0})=(\frac{3}{4(k+2)}+\frac{k+6}{36})v^{1,0},  \] 
            \[ L_0^{(1)}(f(-1)v^{1,1})=(\frac{3}{4(k+2)}+\frac{k+18}{36})f(-1)v^{1,1}.  \]  
        \item   For  $1<i\leqslant k$, 
            \[ L(k,i)^{T_1}(0)=\C v^{i,i} , \ \ \ \  L(k,i)^{T_1}(\frac{1}{3}) = \C e_{-\frac{1}{3}}v^{i,i}=\C v^{i,i-1}, \]
            \[ L(k,i)^{T_1}(\frac{2}{3}) = \C f_{-\frac{2}{3}}v^{i,i} \oplus \C e_{-\frac{1}{3}}v^{i,i-1}=\C f(-1)v^{i,i} \oplus \C v^{i,i-2},\]
            \[ L_0^{(1)}(v^{i,i})=(\frac{i(i+2)}{4(k+2)}+\frac{k-6i}{36})v^{i,i}, \]
            \[ L_0^{(1)}(v^{i,i-1})=(\frac{i(i+2)}{4(k+2)}+\frac{k-6i+12}{36})v^{i,i-1},  \] 
            \[ L_0^{(1)}(v^{i,i-2})=(\frac{i(i+2)}{4(k+2)}+\frac{k-6i+24}{36})v^{i,i-2}.  \]  

    \end{enumerate}
\end{lem}

\begin{lem}
    For $0\leqslant i\leqslant k$, write $L(k,i)^{T_2} = \oplus_{n \in \frac{1}{3}\Z_{\geqslant 0}}L(k,i)^{T_2}(n)$ as an admissible $\si^2$-twisted $L(k,0)$-module. Then 
    \begin{enumerate}
    	\item   For  $i=0$ and $k=1$, 
            \[L(k,i)^{T_2}(0)=\C \1, \ \ \ \  L(k,i)^{T_2}(\frac{1}{3}) = \C f_{-\frac{1}{3}}\1=\C f(-1)\1, \]
            \[L(k,i)^{T_2}(\frac{2}{3}) = \C e_{-\frac{2}{3}}\1 \oplus \C f_{-\frac{1}{3}}^2\1 = 0,  \]
            \[e_{-\frac{5}{3}}\1 = e(-1)\1 \in L(k,i)^{T_2}(\frac{5}{3}),    \]
            \[ L_0^{(2)}(\1)=\frac{1}{9}\1, \ \ \ \  L_0^{(2)}(f(-1)\1)=\frac{4}{9}f(-1)\1,  \]
            \[ L_0^{(2)}(e(-1)\1)=\frac{16}{9}e(-1)\1.   \]
        \item   For  $i=0$ and $k>1$, 
            \[L(k,i)^{T_2}(0)=\C \1, \ \ \ \  L(k,i)^{T_2}(\frac{1}{3}) = \C f_{-\frac{1}{3}}\1=\C f(-1)\1, \]
            \[L(k,i)^{T_2}(\frac{2}{3}) = \C e_{-\frac{2}{3}}\1 \oplus \C f_{-\frac{1}{3}}^2\1 = \C f(-1)^2\1,  \]
            \[ L_0^{(2)}(\1)=\frac{k}{9}\1, \ \ \ \  L_0^{(2)}(f(-1)\1)=(\frac{k}{9}+\frac{1}{3})f(-1)\1,  \]
            \[ L_0^{(2)}(f(-1)^2\1)=(\frac{k}{9}+\frac{2}{3})f(-1)^2\1.   \]
        \item   For  $i=1$ and $k=1$, 
            \[L(k,i)^{T_2}(0)=\C v^{1,1}, \ \ \ \ L(k,i)^{T_2}(\frac{1}{3})=\C f_{-\frac{1}{3}}v^{1,1}=\C f(-1)v^{1,1}= 0,\]
            \[L(k,i)^{T_2}(\frac{2}{3}) = \C e_{-\frac{2}{3}}v^{1,1} \oplus \C f_{-\frac{1}{3}}^2v^{1,1} = \C v^{1,0}, \] 
            \[f_{-\frac{4}{3}}v^{1,1}=f(-2)v^{1,1} \in L(k,i)^{T_2}(\frac{4}{3})    \]
            \[ L_0^{(2)}(v^{1,1})=\frac{1}{36}v^{1,1}, \ \ \ \  L_0^{(2)}(v^{1,0})=\frac{25}{36}v^{1,0}, \]
            \[ L_0^{(2)}(f(-2)v^{1,1})=\frac{49}{36}f(-2)v^{1,1}.   \]
        \item   For  $i=1$ and $k>1$, 
            \[ L(k,i)^{T_2}(0)=\C v^{1,1}, \ \ \ \  L(k,i)^{T_2}(\frac{1}{3})=\C f_{-\frac{1}{3}}v^{1,1}=\C f(-1)v^{1,1},\]
            \[ L(k,i)^{T_2}(\frac{2}{3}) = \C e_{-\frac{2}{3}}v^{1,1} \oplus \C f_{-\frac{1}{3}}^2v^{1,1}=\C v^{1,0} \oplus \C f(-1)^2v^{1,1},\]
            \[ L_0^{(2)}(v^{1,1})=(\frac{3}{4(k+2)}+\frac{k-3}{9})v^{1,1}, \]        
            \[ L_0^{(2)}(f(-1)v^{1,1})=(\frac{3}{4(k+2)}+\frac{k}{9})f(-1)v^{1,1},  \]
            \[ L_0^{(2)}(v^{1,0})=(\frac{3}{4(k+2)}+\frac{k+3}{9})v^{1,0}.  \]   
        \item   For  $1<i<k$, 
            \[ L(k,i)^{T_2}(0)=\C v^{i,i}, \ \ \ \ L(k,i)^{T_2}(\frac{1}{3})=\C f_{-\frac{1}{3}}v^{i,i}=\C f(-1)v^{i,i}, \]
            \[ e_{-\frac{2}{3}}v^{i,i} = v^{i,i-1} \in L(k,i)^{T_2}(\frac{2}{3}), \]
            \[ L_0^{(2)}(v^{i,i})=(\frac{i(i+2)}{4(k+2)}+\frac{k-3i}{9})v^{i,i}, \]
            \[ L_0^{(2)}(f(-1)v^{i,i})=(\frac{i(i+2)}{4(k+2)}+\frac{k-3i+3}{9})f(-1)v^{i,i},  \] 
            \[ L_0^{(2)}(v^{i,i-1})=(\frac{i(i+2)}{4(k+2)}+\frac{k-3i+6}{9})v^{i,i-1}.  \]
        \item   For  $1<i=k$, 
            \[ L(k,i)^{T_2}(0)=\C v^{k,k}, \ \ \ \  L(k,i)^{T_2}(\frac{1}{3}) = \C f_{-\frac{1}{3}}v^{k,k}=0, \]
            \[ L(k,i)^{T_2}(\frac{2}{3}) = \C e_{-\frac{2}{3}}v^{k,k} \oplus \C f_{-\frac{1}{3}}^2v^{k,k}=\C v^{k,k-1},\]
            \[ f_{-\frac{4}{3}}v^{k,k} = f(-2)v^{k,k} \in L(k,i)^{T_2}(\frac{4}{3}), \]
            \[ L_0^{(2)}(v^{k,k})=(\frac{k}{36})v^{k,k}, \ \ \ \ L_0^{(2)}(v^{k,k-1})=(\frac{k}{36}+\frac{2}{3})v^{k,k-1},\]
            \[ L_0^{(2)}(f(-2)v^{k,k})=(\frac{k}{36}+\frac{4}{3})f(-2)v^{k,k}.  \] 
    \end{enumerate}
\end{lem}

Now we are poised to give the classification of the irreducible $L(k,0)^{\Z_3}$-modules coming from $\si^r$-twisted $L(k,0)$-modules $L(k,i)^{T_r} (0 \leqslant i \leqslant k, r = 1, 2)$. Note that $v^{0,0}=\1$. Set
    
    \[u_{k,i}^{T_1,0} = v^{i,i} \in L(k,i)^{T_1}(0), \ \ \ \ \ \ \ \ \   0 \leqslant i \leqslant k   \]

    \[u_{k,i}^{T_1,1} = \begin{cases}
                            e(-1)\1 \in L(k,0)^{T_1}(\frac{4}{3}),    &  \ \   i=0, k \geqslant 1    \\
                            v^{i,i-1} \in L(k,i)^{T_1}(\frac{1}{3}),  &  \ \   0 < i \leqslant k
                      \end{cases}  \]

    \[u_{k,i}^{T_1,2} = \begin{cases}
                            f(-1)\1 \in L(k,0)^{T_1}(\frac{2}{3}),         &    i=0, k \geqslant 1     \\
                            f(-2)v^{1,1} \in L(1,1)^{T_1}(\frac{5}{3}),    &    i=1, k=1               \\
                            f(-1)v^{1,1} \in L(k,1)^{T_1}(\frac{2}{3}),    &    i=1, k>1               \\
                            v^{i,i-2} \in L(k,i)^{T_1}(\frac{2}{3}),       &    1 < i \leqslant k
                      \end{cases} \]

    \[u_{k,i}^{T_2,0} = v^{i,i} \in L(k,i)^{T_2}(0), \ \ \ \ \ \ \ \ \ \ \ \   0 \leqslant i \leqslant k   \]

    \[u_{k,i}^{T_2,1} = \begin{cases}
                            f(-1)v^{i,i} \in L(k,i)^{T_2}(\frac{1}{3}),    &    0 \leqslant i < k    \\
                            f(-2)v^{k,k} \in L(k,k)^{T_2}(\frac{4}{3}),    &    i = k
                      \end{cases} \]   
    
    \[ u_{k,i}^{T_2,2} = \begin{cases}
                            e(-1)\1 \in L(1,0)^{T_2}(\frac{5}{3}),         &  \ \ \   i=0, k = 1     \\
                            f(-1)^2\1 \in L(k,0)^{T_2}(\frac{2}{3}),       &  \ \ \   i=0, k>1               \\
                            v^{i,i-1} \in L(k,i)^{T_2}(\frac{2}{3}),       &  \ \ \   1 \leqslant i \leqslant k.
                      \end{cases} \]

Then we have the following results.

\begin{lem}
Let $L(k,i)^{T_r,j}$ be the $L(k,0)^{\Z_3}$-modules generated by $u_{k,i}^{T_r,j}$, where $k \in \Z_{\geqslant 1}$, $0 \leqslant i \leqslant k$, $r = 1, 2$, $j = 0, 1, 2$. Then $L(k,i)^{T_r,j}$ $(k \in \Z_{\geqslant 1}$, $0 \leqslant i \leqslant k$, $r = 1, 2$, $j = 0, 1, 2)$ are irreducible $L(k,0)^{\Z_3}$-modules. 
\end{lem}
\begin{proof}
Since we write $L(k,i)^{T_r} = \oplus_{n \in \frac{1}{3}\Z_{\geqslant 0}}L(k,i)^{T_r}(n)$ as an admissible $\si^r$-twisted $L(k,0)$-module, then $L(k,i)^{T_r, j}=\oplus_{n \in \frac{j}{3}+\Z}L(k,i)^{T_r}(n)$ is an irreducible $L(k,0)^{\Z_3}$-module for $j = 0, 1, 2$ \cite{DM97}.
\end{proof}

\begin{thm}  \label{twisted decomposition}
For any $0 \leqslant i \leqslant k$, $r = 1, 2$, we have the following inequivalent irreducible $L(k,0)^{\Z_3}$-module decomposition:
\begin{equation}
L(k,i)^{T_r} = \bigoplus_{j=0}^{2}L(k,i)^{T_r,j}.
\end{equation}
\end{thm}
\begin{proof}
For $r= 1, 2$, a basic fact is that $\si^r$ belongs to ${(\Z_3)}_{L(k,i)^{T_r}}$, thus ${(\Z_3)}_{L(k,i)^{T_r}} = \Z_3$ for any $0 \leqslant i \leqslant k$. Then the theorem follows from (\ref{any irr twist mod decomp}), Theorem \ref{DRX17 thm1} and Theorem \ref{DRX17 thm2}.
\end{proof}

We are now in a position to state the main result of this section.

\begin{thm}
    There are exactly $9(k+1)$ irreducible $L(k,0)^{\Z_3}$-modules up to isomorphism. We give these irreducible $L(k,0)^{\Z_3}$-modules with their conformal weights by Table \ref{Table 1} and Table \ref{Table 2}.
       \begin{table}[h]   
            \caption{$k =1$} \label{Table 1}
            \small
            \renewcommand\arraystretch{1.5}
            \begin{center}
                 \begin{tabular}{|c|c|c|c|c|c|c|}
                 \hline
                             & $L(1,0)^0$ & $L(1,0)^1$ & $L(1,0)^2$ & $L(1,1)^0$ & $L(1,1)^1$ & $L(1,1)^2$  \\
                 \hline
                 $\omega$     &   $0$ &   $1$    &   $1$ &  $\frac{1}{4}$ & $\frac{1}{4}$ & $\frac{9}{4}$   \\
                 \hline
                             & $L(1,0)^{T_1,0}$ & $L(1,0)^{T_1,1}$ & $L(1,0)^{T_1,2}$ & $L(1,1)^{T_1,0}$ & $L(1,1)^{T_1,1}$ & $L(1,1)^{T_1,2}$  \\
                 \hline
                 $\omega$     &  $\frac{1}{36}$ &   $\frac{49}{36}$    &   $\frac{25}{36}$ &  $\frac{1}{9}$ & $\frac{4}{9}$ & $\frac{16}{9}$   \\
                 \hline
                             & $L(1,0)^{T_2,0}$ & $L(1,0)^{T_2,1}$ & $L(1,0)^{T_2,2}$ & $L(1,1)^{T_2,0}$ & $L(1,1)^{T_2,1}$ & $L(1,1)^{T_2,2}$  \\
                 \hline
                 $\omega$    &  $\frac{1}{9}$ & $\frac{4}{9}$ &  $\frac{16}{9}$   &  $\frac{1}{36}$ &   $\frac{49}{36}$    &   $\frac{25}{36}$  \\
                 \hline
                 \end{tabular}
            \end{center}
        \end{table}  
        \begin{table}[h]
            \caption{$k > 1$} \label{Table 2}           
            \small
            \renewcommand\arraystretch{1.6}
            \begin{center}
                 \begin{tabular}{|c|c|c|c|}
                 \hline
                    $i=0$    & $L(k,0)^0$ & $L(k,0)^1$ & $L(k,0)^2$   \\
                 \hline
                 $\omega$    &   $0$      &   $1$      &   $1$    \\
                 \hline
                    $i=1$    & $L(k,1)^0$ & $L(k,1)^1$ & $L(k,1)^2$ \\
                 \hline
                 $\omega$    &  $\frac{3}{4(k+2)}$  & $\frac{3}{4(k+2)}$ & $\frac{4k+11}{4(k+2)}$   \\
                 \hline
                 $1<i\leqslant k$   & $L(k,i)^0$ & $L(k,i)^1$ & $L(k,i)^2$  \\
                 \hline
                 $\omega$    &  $\frac{i(i+2)}{4(k+2)}$ &   $\frac{i(i+2)}{4(k+2)}$  &   $\frac{i(i+2)}{4(k+2)}$   \\
                 \hline
                    $i=0$    & $L(k,0)^{T_1,0}$ & $L(k,0)^{T_1,1}$ & $L(k,0)^{T_1,2}$ \\
                 \hline
                 $\omega$    &  $\frac{k}{36}$ & $\frac{k+48}{36}$ & $\frac{k+24}{36}$   \\
                 \hline
                 $0<i\leqslant k$ & $L(k,i)^{T_1,0}$ &   $L(k,i)^{T_1,1}$ &   $L(k,i)^{T_1,2}$  \\
                 \hline
                 $\omega$    &  $\frac{i(i+2)}{4(k+2)}+\frac{k-6i}{36}$ &   $\frac{i(i+2)}{4(k+2)}+\frac{k-6i+12}{36}$ &               $\frac{i(i+2)}{4(k+2)}+\frac{k-6i+24}{36}$   \\
                 \hline
                 $0 \leqslant i< k$ & $L(k,i)^{T_2,0}$ &   $L(k,i)^{T_2,1}$ &   $L(k,i)^{T_2,2}$  \\
                 \hline
                 $\omega$    &  $\frac{i(i+2)}{4(k+2)}+\frac{k-3i}{9}$ &   $\frac{i(i+2)}{4(k+2)}+\frac{k-3i+3}{9}$ &               $\frac{i(i+2)}{4(k+2)}+\frac{k-3i+6}{9}$   \\
                 \hline
                    $i=k$     & $L(k,k)^{T_2,0}$ & $L(k,k)^{T_2,1}$ & $L(k,k)^{T_2,2}$  \\
                 \hline
                 $\omega$    &  $\frac{k}{36}$ & $\frac{k+48}{36}$ & $\frac{k+24}{36}$   \\
                 \hline
                 \end{tabular}
            \end{center}
        \end{table}
\end{thm}
\begin{proof}
    It follows from the Theorem \ref{DRX17 thm3} that all the irreducible modules of the orbifold vertex operator algebra $L(k,0)^{\Z_3}$ come from $\{L(k,i), L(k,i)^{T_r} | r = 1, 2, 0 \leqslant i \leqslant k\}$. Then the theorem follows from Theorem \ref{V decomposition}, Theorem \ref{irr decomposition} and Theorem \ref{twisted decomposition}. 
    For the case of $k=1$, the lowest weight vectors of $L(1,i)^{T_1,j} (i=0, 1, j=0, 1, 2)$ with their lowest weights have been given in \cite{DJ13}.
\end{proof}

\begin{rmk} \label{k=1 V^G iso to V_L}
    For $k=1$, the orbifold vertex operator algebra $L(1,0)^{\Z_3}$ can be realized as the lattice vertex operator algebra $V_{\Z\be}$ associated to the positive definite even lattice $\Z\be$ with $(\be, \be) = 18$ \cite{DJJJY15}. Moreover, it is well known that there are $18$ inequivalent irreducible $V_{\Z\be}$-modules: \{$V_{\Z\be+\frac{s}{18}\be}| 0 \leqslant s < 18$\} \cite{B86}, \cite{FLM88}. 
    Therefore, from \cite{DJJJY15} together with the Proposition 2.15 in \cite{DLM96-1}, we have the following $L(1,0)^{\Z_3}$-module isomorphisms:

    \[L(1,0) \cong V_{\Z\be} \oplus V_{\Z\be+\frac{6}{18}\be} \oplus V_{\Z\be+\frac{12}{18}\be},\]
    \[L(1,0)^0 \cong V_{\Z\be}, \ \ \ \ L(1,0)^1 \cong V_{\Z\be+\frac{6}{18}\be}, \ \ \ \ L(1,0)^2 \cong V_{\Z\be+\frac{12}{18}\be},\]

    \[L(1,0)^{T_1} \cong V_{\Z\be+\frac{1}{18}\be} \oplus V_{\Z\be+\frac{7}{18}\be} \oplus V_{\Z\be+\frac{13}{18}\be},\]
    \[L(1,0)^{T_1,0} \cong V_{\Z\be+\frac{1}{18}\be}, \ \ \ \ L(1,0)^{T_1,1} \cong V_{\Z\be+\frac{7}{18}\be}, \ \ \ \ L(1,0)^{T_1,2} \cong V_{\Z\be+\frac{13}{18}\be},\]

    \[L(1,0)^{T_2} \cong V_{\Z\be+\frac{2}{18}\be} \oplus V_{\Z\be+\frac{8}{18}\be} \oplus V_{\Z\be+\frac{14}{18}\be},\]
    \[L(1,0)^{T_2,0} \cong V_{\Z\be+\frac{2}{18}\be}, \ \ \ \ L(1,0)^{T_2,1} \cong V_{\Z\be+\frac{14}{18}\be}, \ \ \ \ L(1,0)^{T_2,2} \cong V_{\Z\be+\frac{8}{18}\be},\]

    \[L(1,1) \cong V_{\Z\be+\frac{3}{18}\be} \oplus V_{\Z\be+\frac{9}{18}\be} \oplus V_{\Z\be+\frac{15}{18}\be},\]
    \[L(1,1)^0 \cong V_{\Z\be+\frac{15}{18}\be}, \ \ \ \ L(1,1)^1 \cong V_{\Z\be+\frac{3}{18}\be}, \ \ \ \ L(1,1)^2 \cong V_{\Z\be+\frac{9}{18}\be},\]

    \[L(1,1)^{T_1} \cong V_{\Z\be+\frac{4}{18}\be} \oplus V_{\Z\be+\frac{10}{18}\be} \oplus V_{\Z\be+\frac{16}{18}\be},\]
    \[L(1,1)^{T_1,0} \cong V_{\Z\be+\frac{16}{18}\be}, \ \ \ \ L(1,1)^{T_1,1} \cong V_{\Z\be+\frac{4}{18}\be}, \ \ \ \ L(1,1)^{T_1,2} \cong V_{\Z\be+\frac{10}{18}\be},\]

    \[L(1,1)^{T_2} \cong V_{\Z\be+\frac{5}{18}\be} \oplus V_{\Z\be+\frac{11}{18}\be} \oplus V_{\Z\be+\frac{17}{18}\be},\]
    \[L(1,1)^{T_2,0} \cong V_{\Z\be+\frac{17}{18}\be}, \ \ \ \ L(1,1)^{T_2,1} \cong V_{\Z\be+\frac{11}{18}\be}, \ \ \ \ L(1,1)^{T_2,2} \cong V_{\Z\be+\frac{5}{18}\be}.\]    
\end{rmk}

% \begin{rmk}
%     Notice that for $k=2$, $L(2,0)=V_{\Z\gamma}\otimes L_{Vir}(\frac{1}{2},0)\oplus V_{\Z\gamma + \frac{1}{2}\gamma}\otimes L_{Vir}(\frac{1}{2},\frac{1}{2})$  as a $V_{\Z\gamma}$-module
% % and  $L(2,0)^{\langle \sigma \rangle}=V_{\Z\gamma}\otimes L_{Vir}(\frac{1}{2},0)$, 
%     where $(\gm, \gm) = 4$, and $L_{Vir}(\frac{1}{2}, 0)$ is the simple vertex operator algebra associated to the Virasoro algebra with central charge $\frac{1}{2}$. 

% Then it is easy to see that $L(2,0)^{\Z_3}$ is isomorphic to $V_{\Z\gm}^+ \otimes L_{Vir}(\frac{1}{2}, 0)$.  Using the fact that $V_{\Z\gm}^+$ is isomorphic to $L_{Vir}(\frac{1}{2},0)^{\otimes 2}$ \cite{DGH98}, we have $L_{\widehat{\mathfrak{sl}_2}}(2,0)^{\Z_3} \cong L_{Vir}(\frac{1}{2},0)^{\otimes 3}$. Then there are exactly $27$ inequivalent irreducible $L_{Vir}(\frac{1}{2},0)^{\otimes 3}$-modules.

%All the irreducible $L_{Vir}(\frac{1}{2},0)^{\otimes 3}$-modules are $L_{Vir}(\frac{1}{2}, h_1) \otimes L_{Vir}(\frac{1}{2}, h_2) \otimes L_{Vir}(\frac{1}{2}, h_3)$ where $h_1, h_2, h_3 \in \{0, \frac{1}{2}, \frac{1}{16}\}$.
% \end{rmk}

\section{Quantum dimensions and fusion rules for the orbifold vertex operator algebra $L(k,0)^{\Z_3}$}

In this section, we first recall from \cite{DRX17} some results on the quantum dimensions of irreducible $g$-twisted $V$-modules and irreducible $V^G$-modules for $G$ being a finite automorphism group of the vertex operator algebra $V$. Then we compute the quantum dimensions for irreducible modules of the orbifold vertex operator algebra $L(k,0)^{\Z_3}$. Finally, we determine the fusion rules for the orbifold vertex operator algebras $L(k,0)^{\Z_3}$. 

% For any $L(k,0)^{\Z_3}$-module $W$, we drop the subscript $L(k,0)^{\Z_3}$ in the quantum dimension $qdim_{L(k,0)^{\Z_3}}W$ and simply denote $qdimW$ without causing confusion. 
%  

Let $V$ be a vertex operator algebra, $g$ an automorphism of $V$ with order $T$ and $M=\oplus_{n\in \frac{1}{T}\mathbb{Z}_{\geqslant 0}}M_{\lambda + n}$ a $g$-twisted $V$-module. For any homogeneous element $v \in V$ we define a trace function associated to $v$ as follows:
    \begin{equation}
        Z_M(v,q) = tr_Mo(v)q^{L(0)-\frac{c}{24}} = q^{\lambda-\frac{c}{24}}\sum_{n\in \frac{1}{T}\mathbb{Z}_{\geqslant 0}}tr_{M_{\lambda + n}}o(v)q^{n}
    \end{equation}
where $o(v)=v_{wtv-1}$ is the degree zero operator of $v$, $c$ is the central charge of the vertex operator algebra $V$ and $\lambda$ is the conformal weight of $M$. This is a formal power series in variable $q$. It is proved in \cite{DLM98-1}, \cite{Zhu96} that $Z_M(v,q)$ converges to a holomorphic function, denoted by $Z_M(v,\tau)$, in the domain $|q| < 1$ if $V$ is $C_2$-cofinite. Here and below, $\tau$ is in the upper half plane $\mathbb{H}$ and $q = e^{2\pi\sqrt{-1}\tau}$. Note that if $v = \mathbbm{1}$ is the vacuum vector, then $Z_M(\mathbbm{1},q)$ is the formal character of $M$. We simply denote $Z_M(\mathbbm{1},q)$ and $Z_M(\mathbbm{1},\tau)$ by $\chi_M(q)$ and $\chi_M(\tau)$, respectively. $\chi_M(q)$ is called the \emph{character} of $M$.

Let $V$ be a regular and selfdual vertex operator algebra of CFT type and $G$ a finite automorphism group of $V$. Let $g \in G$ and $M$ a $g$-twisted $V$-module. Then $M$ is a finite sum of irreducible $g$-twisted $V$-modules. In particular, each homogeneous subspace of $M$ is finite dimensional. From the above discussion, we know that $\chi_V(\tau)$ and $\chi_M(\tau)$ are holomorphic functions on $\mathbb{H}$. In \cite{DJX13}, the \emph{quantum dimension} of $M$ over $V$ is defined to be
     \[ qdim_VM = \lim_{y \to 0^+}\frac{\chi_M(\sqrt{-1}y)}{\chi_V(\sqrt{-1}y)} = \lim_{q \to 1^-}\frac{ch_qM}{ch_qV} \]
where $y$ is real and positive, $q = e^{2\pi \sqrt{-1}\tau}, \tau = \sqrt{-1}y$.
From \cite{DJX13}, we know that for any $V$-module $M$, $qdim_VM$ always exists and is greater than or equal to $1$ if the weight of each irreducible $V$-module is positive except $V$ itself. It was proved in \cite{DRX17} that for any $g \in G$ and any $g$-twisted $V$-module $M$, $qdim_VM$ always exists and is nonnegative. Also, $qdim_V(M\circ h) = qdim_VM$ for any $h \in G$ and $qdim_VM = qdim_VM'$.

\begin{lem} (\cite{DJX13})
    Let $V$ be a regular and selfdual vertex operator algebra of CFT type, and $M^0 \cong V, M^1, \cdots, M^d$ are all inequivalent irreducible $V$-modules. We also assume that the conformal weights $\lambda_i$ of $M^i$ are positive for all $i > 0$. Then 
        \begin{equation}
            qdim_V(M^i \boxtimes_V M^j) = qdim_VM^i \cdot qdim_VM^j  
        \end{equation}
    for $0 \leqslant i, j \leqslant d$.
\end{lem}

\begin{lem} (\cite{DRX17}) \label{qdim formula}
    Let $V$ be a regular and selfdual vertex operator algebra of CFT type, $G$ a finite automorphism group of $V$, $g \in G$ and $M$ a $g$-twisted $V$-module, $\lambda \in \Lambda_{G_M, \alpha_M}$. If the weight of any irreducible $g$-twisted $V$-module is positive except $V$ itself. Then
        \begin{equation} qdim_{V^G}M = |G| \cdot qdim_VM, \end{equation}
        \begin{equation} qdim_{V^G}M_{\lambda} = [G:G_M] \cdot dimW_{\lambda} \cdot qdim_VM. \end{equation}
    Moreover, $qdim_VM$ takes values in $\{2\cos\frac{\pi}{n}|n\geqslant 3\} \cup [2, \infty)$.
\end{lem}

% From now on, for $1\leqslant i\leqslant k$, we denote $Q_i = \frac{\sin\frac{\pi(i+1)}{k+2}}{\sin\frac{\pi}{k+2}}$ for convenience.

\begin{lem} \label{qdim of irr V mods}
    For $0 \leqslant i \leqslant k$, the quantum dimensions of irreducible $L(k,0)$-modules are
        \begin{equation}  
            qdim_{L(k,0)}L(k,i) = \frac{\sin\frac{\pi(i+1)}{k+2}}{\sin\frac{\pi}{k+2}}. 
        \end{equation}
\end{lem}

\begin{lem} (\cite{DJJJY15})  \label{qdim of twisted V mods}
    For $0 \leqslant i \leqslant k$, the quantum dimensions of $\si$-twisted $L(k,0)$-modules and of $\si^2$-twisted $L(k,0)$-modules are
        \begin{equation} 
            qdim_{L(k,0)}{L(k,i)^{T_r}} = \frac{\sin\frac{\pi(i+1)}{k+2}}{\sin\frac{\pi}{k+2}}, \ \ r = 1, 2. 
        \end{equation}
\end{lem}

Note that $L(k,0)$ satisfiy all the condictions in Lemma \ref{qdim formula}, ${(\Z_3)}_{L(k,i)} = {\Z_3}$ and ${(\Z_3)}_{L(k,i)^{T_r}} = {\Z_3}$ for $0 \leqslant i \leqslant k$, $r = 1, 2$. Using Lemmas \ref{qdim formula}-\ref{qdim of twisted V mods}, we can compute the quantum dimensions of $L(k,0)^{\Z_3}$-modules: 
    \begin{equation}  
         qdim_{L(k,0)^{\Z_3}}L(k,i) =  qdim_{L(k,0)^{\Z_3}}L(k,i)^{T_r} = 3\frac{\sin\frac{\pi(i+1)}{k+2}}{\sin\frac{\pi}{k+2}}, 
    \end{equation}
for $0 \leqslant i \leqslant k$, $r = 1, 2$.  Therefore, we can easily obtain the quantum dimensions of irreducible $L(k,0)^{\Z_3}$-modules.

\begin{thm}
    The quantum dimensions of irreducible $L(k,0)^{\Z_3}$-modules are
        \begin{equation}  
            qdim_{L(k,0)^{\Z_3}}L(k,i)^j = qdim_{L(k,0)^{\Z_3}}L(k,i)^{T_r, j} = \frac{\sin\frac{\pi(i+1)}{k+2}}{\sin\frac{\pi}{k+2}}, 
        \end{equation}
    for $0 \leqslant i \leqslant k$, $r = 1, 2$, $j = 0, 1, 2$.
\end{thm}

It is observed that $qdim_{L(1,0)^{\Z_3}}M = 1$ for any irreducible $L(1,0)^{\Z_3}$-module $M$. As a consequence, all the irreducible $L(1,0)^{\Z_3}$-modules are simple currents \cite{DJX13}.

Let $V$ be a vertex operator algebra with only finitely many irreducible modules, the \emph{global dimension} is defined as $glob(V) = \sum_{M \in Irr(V)}qdim(M)^2$ \cite{DJX13}. 
Assume G is a finite subgroup of $\Aut(G)$, it is proved that $|G|^2glob(V) = glob(V^G)$ \cite{ADJR18}, \cite{DRX17}.
One immediately gets that 
    \[glob(L(k,0)^{\Z_3}) = 9\sum_{i=0}^{k}(\frac{\sin\frac{\pi(i+1)}{k+2}}{\sin\frac{\pi}{k+2}})^2. \]

Now we recall from \cite{TK88} the fusion rules for the simple affine vertex operator algebra $L(k,0)$.

\begin{lem} 
    \begin{equation} 
        L(k,i) \boxtimes_{L(k,0)} L(k,j) = \sum_{\substack{|i-j| \leqslant l \leqslant i + j \\ i + j + l \in 2\mathbb{Z} \\ i + j + l \leqslant 2k}} L(k,l).
    \end{equation}
\end{lem}

The following Lemma follows from \cite{DLM96-1}.

\begin{lem}  \label{intertwin Y(si^r)}
    For $0 \leqslant i, j, l \leqslant k$, $i + j + l \in 2\mathbb{Z}$, $i + j + l \leqslant 2k$, let $\mathscr{Y}(\cdot, z)$ be an intertwining operator of type 
        $\left(\begin{array}{c}
                 \ L(k,l)\ \\
                L(k,i) \ L(k,j)
              \end{array}\right)$. 
    Define $\mathscr{Y}_{\si^r}(\cdot, z)=\mathscr{Y}(\Delta(h^{(r)}, z)\cdot, z)$. Then $\mathscr{Y}_{\si^r}(\cdot, z)$ is an intertwining operator of type 
        $\left(\begin{array}{c}
                 \ L(k,l)^{T_r}\ \\
                L(k,i)   \    L(k,j)^{T_r}
              \end{array}\right)$.
\end{lem}

In order to determine the contragredient modules of irreducible $L(k,0)^{\Z_3}$-modules, we racall from \cite{DLY09} that the irreducible $L(k,0)$-modules $L(k,i)(0 \leqslant i \leqslant k)$ can be realized in the module $V_{L^{\perp}}$ of the lattice vertex operator algebra $V_L$, where $L = \mathbb{Z}\alpha_1 + \cdots + \mathbb{Z}\alpha_k$ with $\langle \alpha_i, \alpha_j \rangle = 2\delta_{i,j}$, and $L^{\perp}$ is the dual lattice of $L$. More precisely, the top level of $L(k,i)$ is an $i + 1$ dimensional vector space which is spanned by $\{ v^{i,j} | 0 \leqslant j \leqslant i\}$ and $v^{i,j}$ has the explicit form in $V_{L^{\perp}}$:
    \begin{equation} \label{vi0 vii lattice expression}
        v^{0,0} = \mathbbm{1}, \ \ \ 
        v^{i,0} = \sum_{\substack{I \subseteq \{1, 2, \cdots, k\} \\ |I|=i}} e^{\frac{\alpha_I}{2}}, \ \ \ 
        v^{i,i} = \sum_{\substack{I \subseteq \{1, 2, \cdots, k\} \\ |I|=i}} e^{-\frac{\alpha_I}{2}}, 
    \end{equation}
    \begin{equation}  \label{vij lattice expression}
        v^{i,j} = \sum_{\substack{I \subseteq \{1, 2, \cdots, k\} \\ |I|=i}}\sum_{\substack{J \subseteq I \\ |J|=j}} e^{\frac{\alpha_{I-J}}{2}-\frac{\alpha_J}{2}},
    \end{equation}  
where $\alpha_I = \sum_{r \in I}\alpha_r$ for a subset $I$ of $\{1, 2, \cdots, k\}$, and the vertex operator associated with $e^{\alpha}$, $\alpha \in L^{\perp}$ is defined on $V_{L^{\perp}}$ by 
    \begin{equation}  \label{lattice vertex operator}
        \mathscr{Y}(e^{\alpha},z) = \exp(\sum^{\infty}_{n=1}\frac{\alpha(-n)}{n}z^n)\exp(\sum^{\infty}_{n=1}\frac{\alpha(n)}{-n}z^{-n})e_{\alpha}z^{\alpha(0)}.
    \end{equation}
Moreover, the operator $\mathscr{Y}$ produces the intertwining operator for $V_L$ of type
    $\left(\begin{array}{c}
                 \ V_{\lambda_1+\lambda_2+L}\          \\    
                V_{\lambda_1+L}    \     V_{\lambda_2+L}
      \end{array}\right)$ for $\lambda_1, \lambda_2 \in L^{\perp}$.

\begin{thm}  \label{contragredient mod}
    For $0 \leqslant i \leqslant k$, $j \in \{0, 1, 2\}$, $k \in \Z_{\geqslant 1}$. 
       \begin{enumerate}
            \item
            If $i \in 3\mathbb{Z}$, then $(L(k,i)^j)' \cong L(k,i)^{\overline{-j}}$ as irreducible $L(k,0)^{\Z_3}$-modules.
            \item
            If $i \in 3\mathbb{Z}+1$, then $(L(k,i)^j)' \cong L(k,i)^{\overline{1-j}}$ as irreducible $L(k,0)^{\Z_3}$-modules.
            \item
            If $i \in 3\mathbb{Z}+2$, then $(L(k,i)^j)' \cong L(k,i)^{\overline{2-j}}$ as irreducible $L(k,0)^{\Z_3}$-modules.
            \item
            $(L(k,i)^{T_1, j})' \cong L(k,k-i)^{T_2, j}$ as irreducible $L(k,0)^{\Z_3}$-modules.
       \end{enumerate}
       %  \begin{enumerate}
       %      \item
       %      If $i \in 3\mathbb{Z}$, $L(k,i)^0$ are selfdual and $(L(k,i)^1)' \cong L(k,i)^2$ as irreducible $L(k,0)^{\Z_3}$-modules.
       %      \item
       %      If $i \in 3\mathbb{Z}+1$, $L(k,i)^2$ are selfdual and $(L(k,i)^0)' \cong L(k,i)^1$ as irreducible $L(k,0)^{\Z_3}$-modules.
       %      \item
       %      If $i \in 3\mathbb{Z}+2$, $L(k,i)^1$ are selfdual and $(L(k,i)^0)' \cong L(k,i)^2$ as irreducible $L(k,0)^{\Z_3}$-modules.
       %      \item
       %      $(L(k,i)^{T_1, j})' \cong L(k,k-i)^{T_2, j}$ as irreducible $L(k,0)^{\Z_3}$-modules.
       % \end{enumerate}
\end{thm}
\begin{proof}
    The contragredient module of $L(1,i)^j$and $L(1,i)^{T_r, j} (i, r = 0, 1, j = 0, 1, 2)$ can be easily determined by using the lattice vertex operator algebra $V_{\Z\be}$ in Remark \ref{k=1 V^G iso to V_L}. Indeed, the contragredient module of $V_{\Z\be+\frac{s}{18}\be}$ is $V_{\Z\be+\frac{18-s}{18}\be}$ for any $0 \leqslant s < 18$.   
    
    Next, we consider the case of $k>1$.
    A basic fact is that if $V$ is a selfdual vertex operator algebra, $(M, Y_M)$ is a $V$-module and $(M', Y_{M'})$ is the contragredient module of $M$, then $V \subseteq M \boxtimes_V M'$. 
    From Theorem \ref{VG prop}, we know that $L(k,0)^{\Z_3}$ is a selfdual vertex operator algebra.
    Note that $v^{i,i-j} \in L(k,i)^{\overline{j}} (j = 0, 1, 2)$ for any $2 \leqslant i \leqslant k$. 
    Since 
        \[v^{0,0} = \1 \in L(k,0)^0 = L(k,0)^{\Z_3} \subseteq L(k,i)^j \boxtimes_{L(k,0)^{\Z_3}} (L(k,i)^j)',\]
    by using (\ref{vi0 vii lattice expression})-(\ref{lattice vertex operator}), we can deduce that $\mathbbm{1}$ can be obtained from $\mathscr{Y}(v^{i,j},z)(v^{i,i-j})$, where $\mathscr{Y}$ is the nonzero intertwining operator for $V_L$ of type 
        $\left(\begin{array}{c}
                 \ V_{\lambda_1+\lambda_2+L}\          \\    
                V_{\lambda_1+L}    \     V_{\lambda_2+L}
               \end{array}\right)$ 
        for $\lambda_1, \lambda_2 \in L^{\perp}$.
    This implies that $v^{i,0} \in (L(k,i)^0)'$ for any $0 \leqslant i \leqslant k$, $v^{i,1} \in (L(k,i)^1)'$ for any $1 \leqslant i \leqslant k$, and $v^{i,2} \in (L(k,i)^2)'$ for any $2 \leqslant i \leqslant k$. 
    It is observed that 
        \[  i!v^{i,0} = e(0)^iv^{i,i}, \ \ (i-1)!v^{i,1} = e(0)^{i-1}v^{i,i}, \ \ (i-2)!v^{i,2} = e(0)^{i-2}v^{i,i}.\]
    Thus $v^{i,0} \in L(k,i)^0$ if $i \in 3\mathbb{Z}$, $v^{i,0} \in L(k,i)^1$ if $i \in 3\mathbb{Z}+1$ and $v^{i,0} \in L(k,i)^2$ if $i \in 3\mathbb{Z}+2$. As a result, $L(k,i)^0$ is selfdual if $i \in 3\mathbb{Z}$, $(L(k,i)^0)' \cong L(k,i)^1$ if $i \in 3\mathbb{Z}+1$ and $(L(k,i)^0)' \cong  L(k,i)^2$ if $i \in 3\mathbb{Z}+2$. Other contragredient modules in 1-3 could be proved using similar arguments. 

    Next we prove $(L(k,i)^{T_1, j})' \cong L(k,k-i)^{T_2, j}$. From the definition of contragredient module, we know that any $g$-twisted $V$-module $M$ and its contragredient module $M'$ ($g^{-1}$-twisted $V$-module) have the same lowest weight. Note that $a_{k,i}^{(1)} = a_{k,k-i}^{(2)}$, where $a_{k,i}^{(r)} ( r = 1, 2)$ is the conformal weight of $L(k,i)^{T_r}$ defined in Lemma \ref{lw a_kir}. Therefore, $(L(k,i)^{T_1, j})' \cong L(k,k-i)^{T_2, j}$ holds for any $0 \leqslant i \leqslant k, j \in \{0, 1, 2\}$.
\end{proof}

\begin{lem}  \label{h^{(3)} h^{(4)} L(k,0)-iso}
    For $0 \leqslant i \leqslant k$, we have the following $L(k,0)$-isomorphisms.
    \begin{enumerate}
    	\item
        $(L(k,i)^{T_1}, Y_{\si}(\Delta(h^{(1)}, z)\cdot, z)) \cong (L(k, i)^{T_2}, Y_{\si^2}(\cdot, z))$;
        \item
        $(L(k,i), Y(\Delta(h^{(3)}, z)\cdot, z)) \cong (L(k, k-i), Y(\cdot, z))$;
        \item
        $(L(k,i), Y(\Delta(h^{(4)}, z)\cdot, z)) \cong (L(k, k-i)^{T_1}, Y_{\si}(\cdot, z))$.
    \end{enumerate}
\end{lem}
\begin{proof}
    Note that $\Delta(h^{(r)}, z)=\Delta(h^{(1)}, z)^r$ for $r \in \Z_{\geqslant 0}$, and therefore the assertion 1 is obvious. From the Lemma 2.6 in \cite{DLM96-1}, we know that $(L(k,i), Y(\Delta(h^{(3)}, z)\cdot, z))$ is an irreducible $L(k,0)$-module with the eigenvalue $a_{k,i}^{(3)}$ of the operator $L_0^{(3)}$ on $v^{i,i}$ defined in Lemma \ref{lw a_kir}. It is observed that $a_{k,i}^{(3)} = a_{k,k-i}^{(0)}$. It is easy to verify that the map $\psi_1$ defined by
    \begin{align} \label{i3 iso (k-i)0}
        \psi_1 : (L(k,i), Y(\Delta(h^{(3)}, z)\cdot, z)) & \longrightarrow (L(k, k-i), Y(\cdot, z)) \\
                              v^{i,i}                  & \longmapsto         v^{k-i,0}.
    \end{align}
    is an $L(k,0)$-isomorphism. 
    Then we can deduce that $(L(k,i), Y(\Delta(h^{(3)}, z)\cdot, z)) \cong (L(k, k-i), Y(\cdot, z))$ as an irreducible $L(k,0)$-isomorphism, i.e., the assertion 2 holds. 
    
    Finally, the assertion 3 is immediate by using the assertion 2:
    \begin{align*}    
        (L(k,i), Y(\Delta(h^{(4)}, z)\cdot, z)) & \cong (L(k, i), Y(\Delta(h^{(1)}, z)\Delta(h^{(3)}, z)\cdot, z))  \\
        &  \cong  (L(k, k-i), Y(\Delta(h^{(1)}, z)\cdot, z)) \\ &  \cong  (L(k, k-i)^{T_1}, Y_{\si}(\cdot, z)). 
    \end{align*} 
    Moreover, we can also construct a $\si$-twisted $L(k,0)$-isomorphism 
    \begin{align} 
        \psi_2 : (L(k,i), Y(\Delta(h^{(4)}, z)\cdot, z)) & \longrightarrow (L(k, k-i)^{T_1}, Y_{\si}(\cdot, z)) \\
                              v^{i,i}                 & \longmapsto    v^{k-i,0}.      \label{i^4 iso (k-i)^1}
    \end{align} 
\end{proof}

The following corollary is clear by noting that $\Delta(h^{(r)}, z)=\Delta(h^{(1)}, z)^r$ for $r \in \Z_{\geqslant 0}$.

\begin{cor}
    For $0 \leqslant i \leqslant k$, we have the following $L(k,0)$-isomorphisms.
    \begin{enumerate}
        \item
        $(L(k,i)^{T_1}, Y_{\si}(\Delta(h^{(2)}, z)\cdot, z)) \cong (L(k, k-i), Y(\cdot, z))$;
        \item
        $(L(k,i)^{T_2}, Y_{\si^2}(\Delta(h^{(1)}, z)\cdot, z)) \cong (L(k, k-i), Y(\cdot, z))$;
        \item
        $(L(k,i)^{T_2}, Y_{\si^2}(\Delta(h^{(2)}, z)\cdot, z)) \cong (L(k, k-i)^{T_1}, Y_{\si}(\cdot, z))$.
    \end{enumerate}
\end{cor}

\begin{lem} \label{L(k,i)^Trj iso}
    For $0 \leqslant i \leqslant k$, $j=0, 1, 2$, we have the following $L(k,0)^{\Z_3}$-isomorphisms.
    \begin{enumerate}
    	\item
        $(L(k,i)^j, Y(\Delta(h^{(1)}, z)\cdot, z)) \cong (L(k, i)^{T_1,j}, Y_{\si}(\cdot, z))$;
        \item
        $(L(k,i)^j, Y(\Delta(h^{(2)}, z)\cdot, z)) \cong (L(k, i)^{T_2,\overline{-j}}, Y_{\si^2}(\cdot, z))$;
        \item
        $(L(k,i)^{T_1, j}, Y_{\si}(\Delta(h^{(1)}, z)\cdot, z)) \cong (L(k, i)^{T_2,\overline{-j}}, Y_{\si^2}(\cdot, z))$;
        \item
        $(L(k,i)^j, Y(\Delta(h^{(3)}, z)\cdot, z)) \cong (L(k, k-i)^{\overline{j+k-i}}, Y(\cdot, z))$;
        \item
        $(L(k,i)^j, Y(\Delta(h^{(4)}, z)\cdot, z)) \cong (L(k, k-i)^{T_1,\overline{j+k-i}}, Y_{\si}(\cdot, z))$;
        \item
        $(L(k, i)^{T_2,j}, Y_{\si^2}(\Delta(h^{(1)}, z)\cdot, z)) \cong (L(k, k-i)^{\overline{-j+k-i}}, Y(\cdot, z))$;
        \item
        $(L(k, i)^{T_2,j}, Y_{\si^2}(\Delta(h^{(2)}, z)\cdot, z)) \cong (L(k, k-i)^{T_1,\overline{-j+k-i}}, Y_{\si}(\cdot, z))$.
    \end{enumerate}
\end{lem}
\begin{proof}
    For $k=1$, these isomorphisms can be easily confirmed by using the lattice vertex operator algebra $V_{\Z\be}$ in Remark \ref{k=1 V^G iso to V_L}. 
    Now we prove the case of $k>1$.

    We first show the assertion 1. From the Lemma 2.6 in \cite{DLM96-1}, we know that $(L(k,i)^{j}$, $Y(\Delta(h^{(1)}, z)\cdot, z))$ is an irreducible $L(k,0)^{\Z_3}$-module. Since $L(k,i)=L(k,i)^{T_1}$ as vector spaces and $Y(\Delta(h^{(1)}, z)\cdot, z)=Y_{\si}(\cdot, z)$ as vertex operators on $L(k,i)$, it follows from $v^{i,i-j} \in L(k,i)^j \cap L(k,i)^{T_1,j}$ that $L(k,i)^{j} = L(k, i)^{T_1,j}$. Then we obtain the assertion 1.
    
    The assertion 2 can be proved by using similar arguments. Just note that $v^{i,i-j} \in L(k,i)^j \cap L(k,i)^{T_2,\overline{-j}}$. Then we obtain the assertion 2.

    The assertion 3 follows from the assertion 1 and 2.

    Next, we show the assertion 4. Recall the $L(k,0)$-isomorphism $\psi_1$ defined in (\ref{i3 iso (k-i)0}) and the fact that $a_{k,i}^{(3)} = a_{k,k-i}^{(0)}$. Moreover, $v^{k-i,0} \in L(k,k-i)^0$ if $k-i \in 3\mathbb{Z}$, $v^{k-i,0} \in L(k,k-i)^1$ if $k-i \in 3\mathbb{Z}+1$ and $v^{k-i,0} \in L(k,k-i)^2$ if $k-i \in 3\mathbb{Z}+2$. 
    Then we can deduce that 
        \begin{align*}
        (L(k,i)^j, Y(\Delta(h^{(3)}, z)\cdot, z)) & \cong (L(k, k-i)^j, Y(\cdot, z)), \ \ \ \ \ \text{if} \ \  k-i \in 3\mathbb{Z},   \\ 
        (L(k,i)^j, Y(\Delta(h^{(3)}, z)\cdot, z)) & \cong (L(k, k-i)^{\overline{j+1}}, Y(\cdot, z)), \ \   \text{if} \ \  k-i \in 3\mathbb{Z}+1,  \\
        (L(k,i)^j, Y(\Delta(h^{(3)}, z)\cdot, z)) & \cong (L(k, k-i)^{\overline{j+2}}, Y(\cdot, z)), \ \   \text{if} \ \  k-i \in 3\mathbb{Z}+2.
        \end{align*}
    This proves the assertion 4.

    % We can prove the assertion 5 by using a similar argument to the the proof of the assertion 4. However, it is worth noting that $a_{k,i}^{(4)} = a_{k,k-i}^{(1)} + \frac{k-i}{3}$.

    Finally, the assertion 5 follows from the assertion 1 and 4, the assertion 6 follows from the assertion 2 and 4, and the assertion 7 follows from the assertion 2 and 5. 
\end{proof}

For $j_1, j_2 \in \Z$, $0 \leqslant i_1, i_2, i_3 \leqslant k$, such that $i_1 + i_2 + i_3 \in 2\mathbb{Z}$, $i_1 + i_2 + i_3 \leqslant 2k$, we define 
    \begin{equation}
        sign(i_1, i_2, i_3, j_1, j_2) =
            \begin{cases}
                j_1 + j_2,  & \text{if}  \quad  \frac{1}{2}(i_1 + i_2 - i_3) \in 3\mathbb{Z}, \\ 
                j_1 + j_2 - 1,  & \text{if}  \quad  \frac{1}{2}(i_1 + i_2 - i_3) \in 3\mathbb{Z}+1, \\
                j_1 + j_2 - 2,  & \text{if}  \quad  \frac{1}{2}(i_1 + i_2 - i_3) \in 3\mathbb{Z}+2.
            \end{cases}
    \end{equation}

Now we are in a position to determine the fusion rules for all the irreducible $L(k,0)^{\Z_3}$-modules. For the irreducible $L(k,0)^{\Z_3}$-modules $W^1$ and $W^2$, we drop the subscript $L(k,0)^{\Z_3}$ in the fusion product $W^1 \boxtimes_{L(k,0)^{\Z_3}} W^2$ and simply denote $W^1 \boxtimes W^2$ without causing confusion. The following theorem together with Proposition \ref{fusionsymm.} and Theorem \ref{contragredient mod} give all the fusion rules for the $\mathbb{Z}_3$-orbifold vertex operator algebra $L(k,0)^{\Z_3}$.

\begin{thm}  \label{fusion rules v^Z3}
The fusion rules for the $\mathbb{Z}_3$-orbifold affine vertex operator algebra $L(k,0)^{\Z_3}$ are as follows:
    \begin{equation}  \label{fusion T0 T0}
        L(k,i_1)^{j_1} \boxtimes L(k,i_2)^{j_2} = \sum_{\substack{|i_1-i_2| \leqslant i_3 \leqslant i_1 + i_2 \\ i_1 + i_2 + i_3 \in 2\Z \\ i_1 + i_2 + i_3 \leqslant 2k}} L(k,i_3)^{\overline{sign(i_1, i_2, i_3, j_1, j_2)}}, 
    \end{equation}
    \begin{equation} \label{fusion T0 T1}
        L(k,i_1)^{j_1} \boxtimes L(k,i_2)^{T_1,j_2} = \sum_{\substack{|i_1-i_2| \leqslant i_3 \leqslant i_1 + i_2 \\ i_1 + i_2 + i_3 \in 2\Z \\ i_1 + i_2 + i_3 \leqslant 2k}} L(k,i_3)^{T_1,\overline{sign(i_1, i_2, i_3, j_1, j_2)}}, 
    \end{equation}
    \begin{equation} \label{fusion T0 T2}
        L(k,i_1)^{j_1} \boxtimes L(k,i_2)^{T_2,j_2} = \sum_{\substack{|i_1-i_2| \leqslant i_3 \leqslant i_1 + i_2 \\ i_1 + i_2 + i_3 \in 2\Z \\ i_1 + i_2 + i_3 \leqslant 2k}} L(k,i_3)^{T_2,\overline{-sign(i_1, i_2, i_3, j_1, -j_2)}}, 
    \end{equation}
    \begin{equation} \label{fusion T1 T1}
        L(k,i_1)^{T_1,j_1} \boxtimes L(k,i_2)^{T_1,j_2} = \sum_{\substack{|i_1-i_2| \leqslant i_3 \leqslant i_1 + i_2 \\ i_1 + i_2 + i_3 \in 2\Z \\ i_1 + i_2 + i_3 \leqslant 2k}} L(k,i_3)^{T_2,\overline{-sign(i_1, i_2, i_3, j_1, j_2)}}, 
    \end{equation}
    \begin{equation}  \label{fusion T1 T2}
        L(k,i_1)^{T_1,j_1} \boxtimes L(k,i_2)^{T_2,j_2} = \sum_{\substack{|i_1-i_2| \leqslant i_3 \leqslant i_1 + i_2 \\ i_1 + i_2 + i_3 \in 2\Z \\ i_1 + i_2 + i_3 \leqslant 2k}} L(k,k-i_3)^{\overline{sign(i_1, i_2, i_3, j_1, -j_2)+k-i_3}}, 
    \end{equation}
    \begin{equation}  \label{fusion T2 T2}
        L(k,i_1)^{T_2,j_1} \boxtimes L(k,i_2)^{T_2,j_2} = \sum_{\substack{|i_1-i_2| \leqslant i_3 \leqslant i_1 + i_2 \\ i_1 + i_2 + i_3 \in 2\Z \\ i_1 + i_2 + i_3 \leqslant 2k}} L(k,k-i_3)^{T_1,\overline{sign(i_1, i_2, i_3, -j_1, -j_2)+k-i_3}}, 
    \end{equation}
where $0 \leqslant i_1, i_2, i_3 \leqslant k$, $j_1, j_2 \in \{0, 1, 2\}$. 
\end{thm}
\begin{proof}
Proof of (\ref{fusion T0 T1}): From Lemma \ref{intertwin Y(si^r)}, we know that $\mathscr{Y}_{\si}(\cdot, z)$ is an intertwining operator of type 
        $\left(\begin{array}{c}
                 \ L(k,i_3)^{T_1}\ \\
                L(k,i_1)   \  L(k,i_2)^{T_1}
              \end{array}\right)$ where $0 \leqslant i_1, i_2, i_3 \leqslant k$, $|i_1-i_2| \leqslant i_3 \leqslant i_1 + i_2$, $i_1 + i_2 + i_3 \in 2\Z$ and $i_1 + i_2 + i_3 \leqslant 2k$.
    Thus we have 
        \[  \mathscr{Y}_{\si}(v^{i_1,i_1}, z)v^{i_2,i_2} = z^{-\frac{i_1}{6}}\mathscr{Y}(v^{i_1,i_1}, z)v^{i_2,i_2}.   \]
    Recall that $a_{k,i}^{(r)} = \frac{i(i+2)}{4(k+2)} + \frac{r^2k-6ir}{36}$ is the conformal weight of the irreducible $\sigma^r$-twisted $L(k, 0)$-module $L(k,i)^{T_r}$ for $r=1, 2$.
    Then we can deduce that $\mathscr{Y}_{\si}(\cdot, z)$ is an intertwining operator of type 
        $\left(\begin{array}{c}
                 \ L(k,i_3)^{T_1,j_3}\ \\
                L(k,i_1)^0   \  L(k,i_2)^{T_1,0}
              \end{array}\right)$ 
    if and only if 
        \[  a_{k,i_1}^{(0)} + a_{k,i_2}^{(0)} - a_{k,i_3}^{(0)} - a_{k,i_1}^{(0)} - a_{k,i_2}^{(1)} + a_{k,i_3}^{(1)} + \frac{i_1}{6} + \frac{j_3}{3} \in \Z \]   
    which is equivalent to $\frac{i_1+i_2-i_3}{6} + \frac{j_3}{3} \in \Z$.
    Hence $\overline{j_3} = 0$, if $\frac{1}{2}(i_1 + i_2 - i_3) \in 3\mathbb{Z}$,     
          $\overline{j_3} = 2$, if $\frac{1}{2}(i_1 + i_2 - i_3) \in 3\mathbb{Z}+1$,     
          and $\overline{j_3} = 1$, if $\frac{1}{2}(i_1 + i_2 - i_3) \in 3\mathbb{Z}+2$.
    
    In general, for any $0 \leqslant i_1, i_2, i_3 \leqslant k$, $|i_1-i_2| \leqslant i_3 \leqslant i_1 + i_2$, $i_1 + i_2 + i_3 \in 2\Z$ and $i_1 + i_2 + i_3 \leqslant 2k$, $j_1, j_2, j_3 \in \{0, 1, 2\}$, $\mathscr{Y}_{\si}(\cdot, z)$ is an intertwining operator of type 
        $\left(\begin{array}{c}
                 \ L(k,i_3)^{T_1,j_3}\ \\
                L(k,i_1)^{j_1}   \  L(k,i_2)^{T_1,j_2}
              \end{array}\right)$ 
    if and only if 
        \[  a_{k,i_1}^{(0)} + a_{k,i_2}^{(0)} - a_{k,i_3}^{(0)} - a_{k,i_1}^{(0)} - a_{k,i_2}^{(1)} + a_{k,i_3}^{(1)} + \frac{i_1}{6} - \frac{j_1}{3} - \frac{j_2}{3} + \frac{j_3}{3} \in \Z \]     
    which is equivalent to $\frac{i_1+i_2-i_3}{6} - \frac{j_1+j_2-j_3}{3} \in \Z$.
    Hence $j_3 = \overline{sign(i_1, i_2, i_3, j_1, j_2)}$. Recall the quantum dimensions of irreducible $L(k,0)^{\Z_3}$-modules along with the fact that 
         \[ \frac{\sin\frac{\pi(i_1+1)}{k+2}}{\sin\frac{\pi}{k+2}} \cdot \frac{\sin\frac{\pi(i_2+1)}{k+2}}{\sin\frac{\pi}{k+2}} = \sum_{\substack{|i_1-i_2| \leqslant i_3 \leqslant i_1 + i_2 \\ i_1 + i_2 + i_3 \in 2\mathbb{Z} \\ i_1 + i_2 + i_3 \leqslant 2k}} \frac{\sin\frac{\pi(i_3+1)}{k+2}}{\sin\frac{\pi}{k+2}}.  \]
    Then we can deduce that (\ref{fusion T0 T1}) holds.

Proof of (\ref{fusion T0 T0}): Note that $v^{i,i-j} \in L(k,i)^j$ if and only if $v^{i,i-j} \in L(k,i)^{T_1,j}$, then by (\ref{fusion T0 T1}), we botain (\ref{fusion T0 T0}).

Proof of (\ref{fusion T0 T2}): Recall from Lemma \ref{L(k,i)^Trj iso}, we know that 
    \[ (L(k,i)^{T_1, j}, Y_{\si}(\Delta(h^{(1)}, z)\cdot, z)) \cong (L(k, i)^{T_2,\overline{-j}}, Y_{\si^2}(\cdot, z)). \]
    Then, as a result of \cite{DLM96-1} Proposition 2.8, we can get (\ref{fusion T0 T2}). Actually, one can also use the symmetric property in Proposition \ref{fusionsymm.}, Theorem \ref{contragredient mod} and (\ref{fusion T0 T1}) to determine the fusion reules $N_{L(k,i_1)^{j_1},L(k,i_2)^{T_2,j_2}}^{L(k,i_3)^{T_2,j_3}}$.

Proof of (\ref{fusion T1 T1}):
    Since \[ L(k,i_1)^{j_1} \boxtimes L(k,i_2)^{T_1,j_2} \cong L(k,i_2)^{T_1,j_2} \boxtimes L(k,i_1)^{j_1}, \]
    we can prove (\ref{fusion T1 T1}) by using (\ref{fusion T0 T1}), the Proposition 2.8 in \cite{DLM96-1} and Lemma \ref{L(k,i)^Trj iso}.

Proof of (\ref{fusion T1 T2}):
    Using (\ref{fusion T1 T1}), the Proposition 2.8 in \cite{DLM96-1} along with Lemma \ref{L(k,i)^Trj iso}, we can deduce that 
        \[   L(k,i_1)^{T_1,j_1} \boxtimes L(k,i_2)^{T_2,\overline{-j_2}} = \sum_{\substack{|i_1-i_2| \leqslant i_3 \leqslant i_1 + i_2 \\ i_1 + i_2 + i_3 \in 2\Z \\ i_1 + i_2 + i_3 \leqslant 2k}} L(k,k-i_3)^{\overline{sign(i_1, i_2, i_3, j_1, j_2)+k-i_3}}.  \] 
    Then (\ref{fusion T1 T2}) is clear.

In almost exactly the same way, we can prove (\ref{fusion T2 T2}).
\end{proof}

% We now give a specific example which is compatible with the Theorem \ref{fusion rules v^Z3}. 
\begin{rmk}
    For the case of $k=1$, recall from Remark \ref{k=1 V^G iso to V_L} that $L(1,0)^{\Z_3}$ can be realized as the lattice vertex operator algebra $V_{\Z\be}$ with $(\be, \be) = 18$ and the correspondence between irreducible $L(1,0)^{\Z_3}$-modules and \{$V_{\Z\be+\frac{s}{18}\be}| 0 \leqslant s < 18$\} has been listed explicitly. It is well known that
        \[ V_{\Z\be+\frac{s}{18}\be} \boxtimes_{V_{\Z\be}} V_{\Z\be+\frac{t}{18}\be} = V_{\Z\be+\frac{s+t}{18}\be}, \]
    where we use $s, t$ to denote both integers between $0$ and $17$ and its residue class modulo $18$ in this situation. This formula also gives the fusion rules for all the irreducible $L(1,0)^{\Z_3}$-modules. It is not difficult to verify that the fusion rules given in this manner are consistent with the results in Theorem \ref{fusion rules v^Z3}.
\end{rmk}

\section*{Acknowledgement}
% The author would like to thank Professor C. Jiang for useful discussions.
% I am deeply appreciated to my supervisor, Professor C. Jiang, who provides her expert guidance, support and encouragement. 
% The author was supported by China NSF grant No.11771281 and SNSFC grant No.16ZR1417800.
I would like to express my deep appreciation and gratitude to my supervisor, Prof. C. Jiang, for her immense knowledge, useful discussions, and valuable suggestions throughout this work.


\begin{thebibliography}{99}

\bibitem{ABD04} T. Abe, G. Buhl, C. Dong, Rationality, regularity, and $C_2$-cofiniteness, \emph{Trans. Amer. Math. Soc.} \textbf{356} (2004) 3391-3402.

\bibitem{ADL05} T. Abe, C. Dong, H. Li, Fusion rules for the vertex operator algebra $M(1)$ and $V^+_L$, \emph{Comm. Math. Phys.} {\bf 253} (2005) 171–219.


% \bibitem{AP} Dra\v{z}en Adamovi\'{c}, Ozern Per\u{s}e, On coset vertex operator algebras with central charge 1, {\em Math. Commun.} {\bf 15} (2010) 143-157.

\bibitem{ADJR18} C. Ai, C. Dong, X. Jiao, L. Ren, The irreducible modules and fusion rules for the parafermion vertex operator algebras, \emph{Trans. Amer. Math. Soc.} \textbf{370} (2018) 5963-5981.


\bibitem{B86} R. E. Borcherds, Vertex algebras, Kac-Moody algebras, and the Monster, \emph{Proc. Natl. Acad. Sci. USA} \textbf{83} (1986) 3068-3071.

\bibitem{CM} S. Carnahan, M. Miyamoto, Regularity of fixed-point vertex operator subalgebras, arXiv: 1603.16045v3.

% \bibitem{CL} T. Chen, C. Lam, Extension of the tensor product of unitary Virasoro vertex operator algebra, \emph{Commun. Alg.} \textbf{35} (2007) 2487-2505.

\bibitem{DG98} C. Dong, R. Griess, Rank one lattice type vertex operator algebras and their automorphism groups, \emph{J. Algebra} \textbf{208} (1998) 262–275.

% \bibitem{DGH98}C. Dong, R. Griess, G. H\"ohn, Framed vertex operator algebras, codes and the moonshine module, \emph{Commun. Math. Phys.} \textbf{193} (1998) 407–448.

\bibitem{DJ13} C. Dong, C. Jiang, Representations of the vertex operator algebra $V_{L_2}^{A_4}$, \emph{J. Algebra} \textbf{377} (2013) 76–96.

\bibitem{DJJJY15} C. Dong, C. Jiang, Q. Jiang, X. Jiao, N. Yu, Fusion rules for the vertex operator algebra $V^{A_4}_{L_2}$, \emph{J. Algebra} \textbf{423} (2015) 476-505.


% \bibitem{DJL} C. Dong, C. Jiang, X. Lin, Rationality of vertex operator algebra $V_L^+:$ higher rank, \emph{Proc. London Math. Soc.} \textbf{104} (2012) 799-826

\bibitem{DJX13} C. Dong, X. Jiao, F. Xu, Quantum dimensions and quantum Galois theory, \emph{Trans. Amer. Math. Soc.} {\bf 365} (2013) 6441–6469.


% \bibitem{DLWY} C. Dong, C. Lam, Q. Wang, H. Yamada, The structure of parafermion vertex operator algebras, {\em J. Algebra} {\bf 323} (2010) 371-381.

\bibitem{DLY09} C. Dong, C. Lam, H. Yamada, $W$-algebras related to parafermion vertex operator algebras, \emph{J. Algebra} \textbf{322} (2009) 2366-2403.

\bibitem{DLM96-1}C. Dong, H. Li, G. Mason, Simple currents and extensions of vertex operator algebras, \emph{Commun. Math. Phys.} \textbf{180} (1996) 671-707.

\bibitem{DLM96-2} C. Dong, H. Li, G. Mason, Compact automorphism groups of vertex operator algebras, \emph{Internat. Math. Res. Notices} \textbf{18} (1996), 913-921.

\bibitem{DLM97-1} C. Dong, H. Li, G. Mason, Regularity of rational vertex operator algebras, \emph{Advances in Math.} \textbf{132} (1997) 148-166.

\bibitem{DLM98-1} C. Dong, H. Li, G. Mason, Twisted representations of vertex operator algebras, \emph{Math. Ann.} \textbf{310} (1998) 571-600.


\bibitem{DLM00} C. Dong, H. Li, G. Mason, Modular invariance of trace functions in orbifold theory and generalized moonshine, \emph{Commu. Math. Phys.} \textbf{214} (2000) 1-56.

\bibitem{DM97} C. Dong, G. Mason, On quantum Galois theory, \emph{Duke Math. J.} \textbf{86} (1997) 305-321.

\bibitem{DN99} C. Dong, K. Nagatomo, Representations of vertex operator algebra $V_L^+$ for rank one lattice $L$, \emph{Comm. Math. Phys.} \textbf{202} (1999) 169–195.


\bibitem{DRX17} C. Dong, L. Ren, F. Xu, On orbifold theory. \emph{Adv. Math.} \textbf{321} (2017) 1-30.

\bibitem{DY02} C. Dong, G. Yamskulna, Vertex operator algebras, generalized double and dual pairs, \emph{Math. Z.} \textbf{241} (2002) 397-423.


% \bibitem{FF85} A. J. Feingold, I. B. Frenkel, Classical affine algebras, \emph{Adv. Math.} \textbf{56} (1985) 117-172.

\bibitem{FHL93} I. B. Frenkel, Y.-Z. Huang, J. Lepowsky, On axiomatic approaches to vertex operator algebras and modules, \emph{Mem. Amer. Math. Soc.} \textbf{104} (1993).


\bibitem{FLM88} I. B. Frenkel, J. Lepowsky, A. Meurman, \emph{Vertex Operator Algebras and the Monster}, Pure Appl. Math., vol. \textbf{134}, Academic Press, Massachusetts, 1988.


\bibitem{FZ92} I. B. Frenkel, Y. Zhu, Vertex operator algebras associated to representations of affine and Virasoro algebra, \emph{Duke Math. J.} \textbf{66} (1992) 123-168.


% \bibitem{GKO85} P. Goddard, A. Kent, D. Olive, Virasoro algebras and coset space models, \emph{Phys. Lett.} \textbf{B152} (1985) 88-92.


% \bibitem{GKO86} P. Goddard, A. Kent, D. Olive, Unitary representations of the Virasoro and super-Virasoro algebra, \emph{Commun. Math. Phys.} \textbf{103} (1986) 105-119.


\bibitem{H95} Y.-Z. Huang, A theory of tensor products for module categories for a vertex operator algebra. IV, \emph{J. Pure Appl. Algebra} \textbf{100} (1995), no. 1-3, 173–216.
% , DOI 10.1016/0022-4049(95)00050-7. MR1344849


\bibitem{HL95-1} Y.-Z. Huang, J. Lepowsky, A theory of tensor products for module categories for a vertex operator algebra. I, II, \emph{Selecta Math. (N.S.)} \textbf{1} (1995), no. 4, 699–756, 757–786.


\bibitem{HL95-2} Y.-Z. Huang, J. Lepowsky, A theory of tensor products for module categories for a vertex operator algebra. III, \emph{J. Pure Appl. Algebra} \textbf{100} (1995), no. 1-3, 141–171.



% \bibitem{JLam19} C. Jiang, C. Lam, Level-rank duality for vertex operator algebras of type $B$ and $D$, \emph{Bull. Inst. Math. Acad. Sin.} (N.S.) \textbf{14} (2019) no. 1, 55–86.

% \bibitem{JLin} C. Jiang, Z. Lin, Tensor decomposition, parafermions, level-rank duality, and reciprocity law for vertex operator algebras, arXiv: 14064191.


% \bibitem{JLin16} C. Jiang, Z. Lin, The commutant of $L_{\widehat{\mathfrak{sl}_2}}(n,0)$ in the vertex operator algebra $L_{\widehat{\mathfrak{sl}_2}}(1,0)^{\otimes n}$, \emph{Adv. Math.} \textbf{301} (2016) 227-257.

\bibitem{JBW} C. Jiang, B. Wang, Representations of the orbifold VOAS $L_{\widehat{\frak{sl}_2}}(k,0)^{K}$ and the commutant VOAS $C_{{L_{\widehat{\mathfrak{so}_m}}(1,0)}^{\otimes 3}}({L_{\widehat{\mathfrak{so}_m}}(3,0)})$, arXiv:1909.08173v2.

\bibitem{JWQ19} C. Jiang, Q. Wang, Representations of $\mathbb{Z}_2$-orbifold of the parafermion vertex operator algebra $K(\mathfrak{sl}_2, k)$, \emph{J. Algebra} \textbf{529} (2019) 174–195.


\bibitem{JW2} C. Jiang, Q. Wang, Fusion rules for $\mathbb{Z}_2$-orbifolds of affine and parafermion vertex operator algebras, arXiv:1904.01798.


% \bibitem{KW94} V. G. Kac, W. Wang, Vertex operator superalgebras and their representations, \emph{Contemp. Math. Amer. Math. Soc.} \textbf{175} (1994) 161-191.

% \bibitem{KLY} M. Kitazume, C. Lam, H. Yamada, Decomposition of the moonshine vertex operator algebra as Virasoro modules,  \emph{J. Algebra}  \textbf{226} (2000) 893-919.

% \bibitem{KMY}  M. Kitazume, M. Miyamoto, H. Yamada, Ternary codes and vertex operator algebras,  {\em J. Algebra} {\bf 223} (2000) 379-395.

% \bibitem{Lam01} C. Lam, Induced module for orbifold vertex operator algebras, {\em J. Math. Soc. Japan} {\bf 53} (2001) 541-557.

% \bibitem{Lam14} C. Lam, A level-rank duality for parafermion vertex operator algebras of type A, \emph{Proc. Amer. Math. Soc.} \textbf{142}(12) (2014) 4133-4142.


\bibitem{LeL04} J. Lepowsky, H. Li, \emph{Introduction to Vertex Operator Algebras and Their Representations}, Progress in Math., Vol. 227, Birkh\"auser, Boston, 2004.


% \bibitem{Li96-1} H. Li, Local systems of vertex operators, vertex superalgebras and modules, \emph{J. Pure Applied Algebra} \textbf{109} (1996) 143-195 .



\bibitem{Li96-2} H. Li, Local systems of twisted vertex operators, vertex operator superalgebras and twisted modules, \emph{Contemp. Math.} \textbf{193} (1996) 203–236.


% \bibitem{Li97-1} H. Li, Extension of vertex operator algebras by a self-dual simple module, \emph{J. Algebra} \textbf{187} (1997) 236-267.


\bibitem{Li97-2} H. Li, The physics superselection principle in vertex operator algebra theory. \emph{J. Algebra} \textbf{196}(2) (1997) 436-457.


\bibitem{M15} M. Miyamoto, $C_2$-cofiniteness of cyclic-orbifold models, \emph{Comm. Math. Phys.} \textbf{335} (2015) 1279-1286.


\bibitem{MT04} M. Miyamoto, K. Tanabe, Uniform product of $A_{g,n}(V)$ for an orbifold model $V$ and $G$-twisted Zhu algebra, \emph{J. Algebra} \textbf{274} (2004) 80-96.

\bibitem{TK88} A. Tsuchiya, Y. Kanie, \emph{Vertex operators in conformal field theory on $\mathbb{P}^1$ and monodromy representations of braid group}, Conformal Field Theory and Solvable Lattice Models, Adv. Studies in Pure Math. vol. 16, Academic Press, New York, 1988, 297-372.

\bibitem{XF00} F. Xu, Algebraic orbifold conformal field theories, \emph{Proc. Natl. Acad. Sci. USA} \textbf{97} (2000) 14069–14073.


\bibitem{XX98} X. Xu, \emph{Introduction to vertex operator superalgebras and their modules}, Mathematics and its applications, Kluwer Academic Publishers, 1998.


\bibitem{Zhu96} Y. Zhu, Modular invariance of characters of vertex operator algebras, \emph{J. Amer. Math. Soc.} \textbf{9} (1996) 237-302.



\end{thebibliography}
\end{document}